\documentclass[11pt,reqno]{amsart}
\usepackage{amsfonts,amssymb,amsmath,latexsym,indentfirst,cite}
\usepackage{amsthm}
\usepackage{eucal}
\newcommand{\R}{\mathbb R}

\newcommand{\N}{\mathbb N}

\newcommand{\KdV}{\textrm{KdV}}

\def\Re{ \mathrm{Re}\, }

\newcommand{\peq}{\hspace*{0.10in}}
\newcommand{\peqq}{\hspace*{0.05in}}

\numberwithin{equation}{section}

\newtheorem{theorem}{Theorem}[section]
\newtheorem{proposition}[theorem]{Proposition}
\newtheorem{remark}[theorem]{Remark}
\newtheorem{lemma}[theorem]{Lemma}
\newtheorem{corollary}[theorem]{Corollary}
\newtheorem{definition}[theorem]{Definition}

\begin{document}
\vglue-1cm \hskip1cm
\title[Scattering for the gKdV equation]{Large data scattering for the defocusing supercritical generalized  KdV equation}

%\author[]{}
%\address{}
%\email{}
%\thanks{}
%\begin{abstract}
%\end{abstract}

%\maketitle

\author[L. G. Farah]{Luiz G. Farah}
\address{ICEx, Universidade Federal de Minas Gerais, Av. Ant\^onio Carlos, 6627, Caixa Postal 702, 30123-970,
Belo Horizonte-MG, Brazil}
\email{lgfarah@gmail.com}

\author[F. Linares]{Felipe Linares}
\address{IMPA, Estrada Dona Castorina 110, CEP 22460-320, Rio de Janeiro, RJ,
 Brazil}
\email{linares@impa.br}
%\thanks will become a 1st page footnote.
%\thanks{The authors were partially supported by CNPq/Brazil and FAPEMIG/Brazil.}
%. The third author was supported by a NSF grant.}

\author[A. Pastor]{Ademir Pastor}
\address{IMECC-UNICAMP, Rua S\'ergio Buarque de Holanda, 651, 13083-859, Cam\-pi\-nas-SP, Bra\-zil}
\email{apastor@ime.unicamp.br}

\author[N. Visciglia]{Nicola Visciglia}
\address{Universit\`a Degli Studi di Pisa, Dipartimento di Matematica, Largo Bruno Pontecorvo 5,  56127  Pisa, Italy}
\email{nicola.visciglia@unipi.it}

%%%%%%%%%%%%%%%%%%%%%%%%%%%%%%%%%%%%%%%%%%%%%%%%%%%%%%%%%%%%%%%%%%%%%%%%%%%%%%%%%%%%%%%%%%%%%%%%

\begin{abstract}
We consider the defocusing supercritical generalized Korteweg-de Vries (gKdV) equation
 $\partial_t u+\partial_x^3u-\partial_x(u^{k+1})=0$,
where $k>4$ is an even integer number. We show that if the initial data $u_0$ belongs to
$H^1$  then the corresponding solution  is global   and scatters in $H^1$.  Our method of proof is inspired on the compactness method introduced by C. Kenig and F. Merle.
\end{abstract}

\maketitle

\section{Introduction}\label{introduction}

We consider solutions of the Initial Value Problem (IVP) associated with the defocusing supercritical generalized Korteweg-de Vries
(gKdV) equation, i.e.,
\begin{equation}\label{gkdv}
\begin{cases}
\partial_t u+\partial_{xxx}u-\partial_x(u^{k+1})=0, \;\;x\in\R, \;t>0, \\
u(0,x)=u_0(x),
\end{cases}
\end{equation}
where $k>4$ is an even integer number.

It is well known that in the focusing case (with the sign `+' in front of the nonlinearity) for $k=1$ and $k=2$ the model above is physical and corresponds to the famous Korteweg-de Vries (\cite{KdV95}) and modified Korteweg-de Vries equation, respectively.

Local well-posedness issues for the IVP \eqref{gkdv} have been of intense study for the last 30 years or more. We refer the reader to Kenig, Ponce, and Vega \cite{kpv1}, \cite{kpv2} for a complete set of sharp results (see also \cite{LP} and references therein). We observe that in the literature the focusing case is always considered but the results in the defocusing case are analogous.  We also notice that in the focusing case
there exist solitary wave solutions  for the equation in  \eqref{gkdv} which is not  true in the other situation.

Here we are mainly interested in the asymptotic behavior of global solutions of the IVP \eqref{gkdv}. 
We recall that the flow of the generalized KdV is conserved by the following quantities
\begin{equation}\label{MC}
\text{Mass}\equiv M[u(t)]=\int u^2(t)\, dx
\end{equation}
and
\begin{equation}\label{EC}
\text{Energy}\equiv E[u(t)]=\frac{1}{2}\int(\partial_xu)^2(t)\;dx+\frac{1}{k+2}\int u^{k+2}(t)\;dx.
\end{equation}

These conserved quantities are fundamental to obtain global well-posedness in the energy $H^1$ since sharp local well-posedness for $k>4$ is
established in Sobolev spaces of order $s>s_k:=\frac12-\frac{2}{k}$ (see \cite{BKPSV96}).  Global well-posedness
results for the IVP   \eqref{gkdv}  for different values of $k$ in the focusing/defocusing cases can be found in \cite{CKSTT3}, \cite{LG6}, \cite{FLP}, 
\cite{FLP1}, \cite{Gr05}, \cite{GPS}, \cite{Gu}, \cite{marmer}, \cite{miao}, \cite{tao1} and references therein.

In the defocusing case and $k>4$, we have the following result.

\begin{theorem}[Global well-posedness for the defocusing gKdV]\label{global5}
	Let $k>4$ be even.
 If $u_0\in H^1$ then the corresponding solution $u$ of \eqref{gkdv} is global in  $H^1$. Moreover, 
 \begin{equation}\label{Apriori}
 \sup_{t\in \R}\|\partial_xu(t)\|^2_{L^2_x}\leq 2E[u_0].
 \end{equation}
 \end{theorem}
 
 This is a consequence of the local well-posedness in $H^1$ result proved in \cite{kpv1} and the conservation of mass \eqref{MC} and energy \eqref{EC}. Indeed, for any local solution we can control its $H^1$-norm since the energy controls the $L^2$-norm of the gradient and the mass is preserved.

Our aim in this paper is to prove that the solutions provided by Theorem \ref{global5} scatter  in $H^1$.  Roughly, we will show that solution of the nonlinear problem behave asymptotically at infinity like a solution of the associated linear problem in $H^1$.

 Before state our main result we will describe some previous nonlinear scattering results for the solutions of the IVP \eqref{gkdv}. In \cite{PV}, Ponce and Vega showed that, for small data, solutions of the focusing  gKdV equation scatter in $H^1$, for $k>4$ without extra conditions on the data.  This result is still valid in the focusing case.  In \cite{kpv1}, Kenig, Ponce and Vega  show that for small data
in $L^2$, solutions of the gKdV, $k=4$, scatters in $L^2$.  For $k=3$ Koch and Marzuola in \cite{koch-m} proved scattering  for small data in $\dot{H}^{-\frac16}$ (see also \cite{tao}). 

For large data, Cot\^e \cite{RC} constructed the wave operator in $H^1$ (respectively $L^2$) for the focusing case and $k>4$ (respectively $k=4$). This  is  the reciprocal  problem  of  the  scattering  theory,  which  consists  in  constructing  a  solution  with  a prescribed scattering state.This result was extended by Farah and Pastor in \cite{FP}, for $k \ge 4$, where the authors showed the existence of wave operator in the critical space $\dot{H}^{s_k}$.

Kenig and Merle \cite{kenig-merle}, \cite{kenig-merle-nlw} introduced a different point of view to study scattering and asymptotic behavior of nonlinear dispersive equations. Their results  open a wide spectrum
of possibilities to attack these kind of problems. Our purpose here is to develop a parallel theory applied to the  defocusing gKdV. More
precisely,
\begin{theorem}\label{global4}
Let $k>4$ be even.
If $u_0\in H^1(\R)$ then the corresponding global solution $u$ of \eqref{gkdv} given by Theorem \ref{global5} scatters in both directions, that is, there exist $\phi^{+}, \phi^{-}\in H^1(\R)$ such that
\begin{equation*}
\lim_{t\rightarrow +\infty} \|u(t)-U(t)\phi^+\|_{H^1}=0
\end{equation*}
and
\begin{equation*}
\lim_{t\rightarrow -\infty} \|u(t)-U(t)\phi^-\|_{H^1}=0.
\end{equation*}
\end{theorem}

Here, $U(t)$ denotes the group of linear operators associated with the linear
KdV equation, that is, for any $u_0\in H^s(\R)$, $s\in \R$, $U(t)u_0$ is the solution of the linear problem
\begin{equation*}%\label{lineargkdv}
\begin{cases}
\partial_t u+\partial_{xxx}u=0, \;\;x\in\R, \;t\in\R, \\
u(0,x)=u_0(x).
\end{cases}
\end{equation*}

To prove Theorem \ref{global4}, we follow very closely the arguments given by Kenig and Merle  in \cite{kenig-merle} and \cite{kenig-merle2}.

We first establish a perturbation theory and a scattering criteria (see \cite{KPV6} and references therein). Then we obtain a profile decomposition. More precisely, we establish that any bounded sequence in $H^1$ has a decomposition into linear profiles and a reminder with a suitable asymptotic smallness property. To do this we follow \cite{dhr} and \cite{CFX}  (see also\cite{PG}, \cite{FP16}). Next we combine this decomposition with a compactness argument in the same lines of Kenig and Merle to produce a critical element for which scattering fails and which enjoys a compactness property. 
However our situation is different from the critical case dealt in \cite{kenig-merle2}; in our analysis the defocusing nature of the equation is taken in consideration. Finally, we show that the critical solution cannot
exist.  This will imply our result.  To complete our argument we need to establish a rigidity result. Due to the lack of useful virial estimates for gKdV we obtain a suitable version of the interaction Morawetz type estimates adapted to the gKdV equation. This approach is inspired in the works of Dodson \cite{dodson}, Colliander et al \cite{ckstt-nls}, Tao \cite{tao1} and references therein.  To obtain our rigidity theorem we combine
the monotonocity formula given in \cite{tao1}, the Kato smoothing effect to justify the regularity of the solutions involved in certain quantities and the persistence of regularity of the flow associated to solutions of the IVP \eqref{gkdv}.

The plan of this paper is as follows. In the next section we introduce some notation and estimates needed in the reminder of the paper.
In Section \ref{MofR}, we establish some perturbation results and a scattering criteria. Next, in Section \ref{sec4}, we show a profile decomposition theorem and, in Section \ref{sec5}, we construct a critical solution. Finally, Section \ref{sec6} is devoted to prove a rigidity theorem and to complete the proof of Theorem \ref{global4}.

%%%%%%%%%%%%%%%%%%%%%%%%%%%%%%%%%%%%%%%%%%%%%%%%%%%%%%%%%%%%%%%%%%%%%%%%%%%%%%%%%%%%%%%%%%%%%%%%%%%%%%%%%%%%%%%%%%%%
%%%%%%%%%%%%%%%%%%%%%%%%%%%%%%%%%%%%%%%%%%%%%%%%%%%%%%%%%%%%%%%%%%%%%%%%%%%%%%%%%%%%%%%%%%%%%%%%%%%%%%%%%%%%%%%%%%%%
%%%%%%%%%%%%%%%%%%%%%%%%%%%%%%%%%%%%%%%%%%%%%%%%%%%%%%%%%%%%%%%%%%%%%%%%%%%%%%%%%%%%%%%%%%%%%%%%%%%%%%%%%%%%%%%%%%%%

\section{Notation and Preliminaries}\label{notation}

Let us start this section by introducing the notation employed throughout the paper. We use $c$ to denote various constants that may
vary line by line. Given any positive numbers $a$ and $b$, the notation $a \lesssim b$ means that there exists a positive constant
$c$ such that $a \leq c\,b$. We use $\|\cdot\|_{L^p}$ to denote the standard $L^p(\R)$ norm. If necessary, we use subscript to inform which variable we are concerned with.
For a given interval $I\subset \R$, the mixed norms $L^q_tL^r_x$ and $L^q_{I}L^r_x$  of $f=f(t,x)$ are defined, respectively, as
\begin{equation*}
\|f\|_{L^q_tL^r_x}= \left(\int_{-\infty}^{+\infty} \|f(t,\cdot)\|_{L^r_x}^q dt \right)^{1/q}\!\!\!\!, \quad \|f\|_{L^q_IL^r_x}= \left(\int_I \|f(t,\cdot)\|_{L^r_x}^q dt \right)^{1/q}
\end{equation*}
with the usual modifications whether $q =\infty$ or $r=\infty$. We use $L^p_{x,t}$ if $p=q$.  In a similar way we define the norm in the spaces  $L^r_xL^q_t$ and $L^r_xL^q_{I}$. If no confusion is caused, we use $\int fdx$ instead of $\int_{\R}f(x)dx$.

 We
shall  define $D^s_x$ and $J^s_x$ to be, respectively, the Fourier
multipliers with symbol $|\xi|^s$ and $\langle \xi \rangle^s = (1+|\xi|)^s$.
The norm in the Sobolev spaces $H^s=H^s(\R)$ and $\dot{H}^s=\dot{H}^s(\R)$
are given, respectively, by
\begin{equation*}
\|f\|_{H^s}\equiv \|J^sf\|_{L^2_x}=\|\langle \xi
\rangle^s\hat{f}\|_{L^2_{\xi}}, \qquad \|f\|_{\dot{H}^s}\equiv
\|D^sf\|_{L^2_x}=\||\xi|^s\widehat{f}\|_{L^2_{\xi}},
\end{equation*}
where $\widehat{f}$ stands for the usual Fourier transform of $f$.

Let us present now some useful lemmas and inequalities. We begin with the following oscillatory integral result concerning the evolution of the linear KdV equation.

\begin{lemma}\label{lemma0} 
For any $\theta \in [0,1]$ we have
\begin{equation*}%\label{eq00}
\|D^{\theta/2}_x U(t)u_0\|_{L^{2/(1-\theta)}_x}\lesssim |t|^{-\theta/2} \|u_0\|_{L^{2/(1+\theta)}_x}.
\end{equation*}
\end{lemma}
\begin{proof}
See \cite[Theorem 2.2]{KPV4}.
\end{proof}

Next we recall the well known Strichartz estimates.

\begin{lemma}\label{lemma1} Let $p,q$, and $\alpha$ be such that
\begin{equation*}%\label{adm}
-\alpha+\dfrac{1}{p}+\dfrac{3}{q}=\dfrac{1}{2}, \quad -\dfrac{1}{2}\leq\alpha\leq\dfrac{1}{q}.
\end{equation*}
Then
\begin{equation}\label{reg1}
 \|D^{\alpha}_xU(t)u_0\|_{L_t^qL_x^p}\lesssim\|u_0\|_{L^2_x}.
\end{equation}
Also, if
\begin{equation}\label{eq3}
\frac{1}{p}+\frac{1}{2q}=\frac{1}{4}, \quad \alpha=\frac{2}{q}-\frac{1}{p},
\quad 1\leq p,q\leq\infty, \quad -\frac{1}{4}\leq\alpha\leq1.
\end{equation}
Then,
\begin{equation}\label{eq4}
\|D^{\alpha}_x U(t)u_0\|_{L^{p}_xL^{q}_t}\lesssim
\|u_0\|_{L^2_x}.
\end{equation}
\end{lemma}
\begin{proof}
See \cite[Theorem 2.1]{KPV4} and \cite[Proposition 2.1]{KPV6}.
\end{proof}

\begin{lemma}\label{kpvlemma}
Assume
\begin{equation}\label{eq1}
\frac{1}{4}\leq\alpha<\frac{1}{2}, \quad
\frac{1}{p}=\frac{1}{2}-\alpha.
\end{equation}
Then,
\begin{equation*}%\label{eq2}
\|D^{-\alpha} U(t)u_0\|_{L^{p}_xL^\infty_t}\lesssim
\|u_0\|_{L^2_x}.
\end{equation*}
Moreover, the smoothing effect of Kato type holds, i.e.
\begin{equation}\label{reg11}
 \|\partial_xU(t)u_0\|_{L_x^{\infty}L_t^2}\lesssim\|u_0\|_{L^2_x}.
\end{equation}
In particular, the dual version of \eqref{reg11} reads as follows:
\begin{equation*}%\label{IE1}
\|\partial_x \int_0^tU(t-t')g(t',\cdot)dt'\|_{L^{2}_x}\lesssim
\| g\|_{L^{1}_xL^{2}_t}.
\end{equation*}
\end{lemma}
\begin{proof}
See \cite[Lemma 3.29]{kpv1} and \cite[Theorem 3.5]{kpv1}.
\end{proof}

%\begin{remark}

%\end{remark}

We can also obtain the following particular cases of the Strichartz estimates in the critical
Sobolev space $\dot{H}^{s_k}(\R)$

\begin{corollary}%\label{corollary12} 
Let $k>4$ and $s_k=(k-4)/2k$. Then
\begin{equation}\label{STR}
 \|D^{2/3k}_xU(t)u_0\|_{L_{x,t}^{3k/2}}\lesssim\|D_x^{s_k}u_0\|_{L^2_x}
\end{equation}
and
\begin{equation}\label{STR2}
 \|D^{-1/k}_xU(t)u_0\|_{L^{k}_xL^{\infty}_t}\lesssim\|D_x^{s_k}u_0\|_{L^2_x}.
\end{equation}
\end{corollary}
\begin{proof}
The estimate \eqref{STR} follows by Sobolev's embedding and Lemma
\ref{lemma1}. On the other hand, estimate \eqref{STR2} is a particular case
of \eqref{eq4}.
\end{proof}

Our next lemma provides another estimate for the linear evolution. It was used by Farah and Pastor \cite{FP} to give a new proof of the local well-posedness, in the critical Sobolev space $H^{s_k}(\R)$, for the gKdV equation \eqref{gkdv} with $k>4$.

\begin{lemma}\label{lemma12} Let $k>4$, $s_k=(k-4)/2k$.
Then
\begin{equation}\label{STR3}
 \|U(t)u_0\|_{L^{5k/4}_xL^{5k/2}_t}\lesssim\|D_x^{s_k}u_0\|_{L^2_x}.
\end{equation}
\end{lemma}
\begin{proof}
It is sufficient to interpolate inequalities \eqref{STR} and \eqref{STR2}.
\end{proof}

We also have the following integral estimates.
\begin{lemma}\label{lemma2}
Let $p_i,q_i$, $\alpha_i$, $i=1,2$, satisfying \eqref{eq3}. Then
$$
\|D_x^{\alpha_1} \int_0^tU(t-t')g(t',\cdot)dt'\|_{L^{p_1}_xL^{q_1}_t}\lesssim
\|D_x^{-\alpha_2} g\|_{L^{p_2'}_xL^{q_2'}_t},
$$
where $p_2'$ and $q_2'$ are the H\"older conjugate of $p_2$ and $q_2$ respectively.

Moreover, assume $(p_1,q_1,\alpha_1)$ satisfies \eqref{eq3} and
$(p_2, \alpha_2)$ satisfies \eqref{eq1}. Then,
\begin{equation*}%\label{eq5}
    \|D_x^{-\alpha_2}\int_0^tU(t-t')f(t',\cdot)dt'\|_{L^{p_2}_xL^\infty_t}\lesssim
    \|D_x^{-\alpha_1}f\|_{L^{p_1'}_xL^{q_1'}_t}.
\end{equation*}
\end{lemma}
\begin{proof}
Use	 duality and $TT^*$ arguments combined with Lemmas
\ref{lemma1} and \ref{kpvlemma} (see also \cite[Proposition 2.2]{KPV6}).
\end{proof}

Applying the same ideas as in the previous lemma together with Lemma \ref{lemma12} and \eqref{reg11} we also have.
\begin{corollary}%\label{corollary21}
Let $k\geq4$ and $s_k=(k-4)/2k$, then the following estimate holds:
$$
\|\partial_x \int_0^tU(t-t')g(t',\cdot)dt'\|_{L^{5k/4}_xL^{5k/2}_t}\lesssim
\|D_x^{s_k} g\|_{L^{1}_xL^{2}_t}.
$$
\end{corollary}

\begin{remark}\label{infinity}
In the above result,  one can replace the integral $\int_0^t$ by $\int_t^\infty$. This kind of
estimate will be used in Section \ref{MofR}.
\end{remark}

Finally, we recall the fractional Leibnitz and chain rules.
\begin{lemma}\label{lemma6}
Let $0<\alpha<1$ and $p,p_1,p_2,q,q_1,q_2 \in (1,\infty)$ with
$\frac{1}{p}=\frac{1}{p_1}+\frac{1}{p_2}$ and
$\frac{1}{q}=\frac{1}{q_1}+\frac{1}{q_2}$. Then,
\begin{itemize}
    \item[(i)] $$
\|D^\alpha_x(fg)-fD^\alpha_xg-gD^\alpha_xf\|_{L^p_xL^q_t}\lesssim
\|D^\alpha_xf\|_{L^{p_1}_xL^{q_1}_t}\|g\|_{L^{p_2}_xL^{q_2}_t}.
$$
The same still holds if $p=1$ and $q=2$.
    \item[(ii)] $$\|D^\alpha_xF(f)\|_{L^p_xL^q_t}\lesssim
    \|D^\alpha_xf\|_{L^{p_1}_xL^{q_1}_t}\|F'(f)\|_{L^{p_2}_xL^{q_2}_t}.$$
  \end{itemize}
\end{lemma}
\begin{proof}
See  \cite[Theorems A.6, A.8, and A.13]{kpv1}.
\end{proof}

\section{Sufficient conditions for scattering and perturbation theory}\label{MofR}

In order to prove our main result, in this section we establish some useful results. 
We start  with the small data theory for the gKdV equation \eqref{gkdv}.

\begin{proposition}[Small data theory]
\label{SDGT}
Let $k \geq 4,\ s_{k} = (k-4)/2k,\ u_{0} \in \dot{H}^{s_{k}}(\R)$ with $\| u_{0} \|_{\dot{H}^{s_{k}}} \leq K$, and $t_{0} \in I$, a time interval. There exists $\delta = \delta(K) > 0$ such that if
\begin{equation*}
\| U(t-t_{0}) u_{0} \|_{L^{5k/4}_{x} L^{5k/2}_{I}} < \delta,
\end{equation*}
there exists a unique solution $u$ of the integral equation
\begin{equation} \label{IntEq}
u(t) = U(t-t_{0}) u_{0} + \int_{t_{0}}^{t} U(t -t') \partial_{x} (u^{k+1})(t') dt'
\end{equation}
in $I \times \R$ with $u \in C(I; \dot{H}^{s_{k}}(\R))$ satisfying
\begin{equation*}%\label{i0}
\| u \|_{L^{5k/4}_{x} L^{5k/2}_{I}} \leq 2\delta \qquad \text{and} \qquad \| u \|_{L^{\infty}_{I} \dot{H}^{s_{k}}_{x}} + \| D^{s_{k}}_{x} u \|_{L^{5}_{x} L^{10}_{I}} < 2cK,
\end{equation*}
for some positive constant $c$.
\end{proposition}
\begin{proof}
The proof can be found in Farah and Pigott \cite[Theorem 3.6]{FP16} (see also Farah and Pastor \cite[Theorem 1.2]{FP} and the proof of Proposition \ref{STPT} below).
\end{proof}

\begin{remark}\label{Rema-SDT}
Let $u_0\in {H}^{1}(\mathbb{R})$ and let $u(t)$ be the corresponding global solution satisfying $$
\displaystyle \sup_{t\in \R}\|u(t)\|_{H^1}:=K<\infty,
$$ 
by Theorem \ref{global5}. We claim that the small data theory (Proposition \ref{SDGT}) implies, for every $a<b$, that
$$
\| u \|_{L^{5k/4}_{x} L^{5k/2}_{[a,b]}} <\infty.
$$
Indeed, let $s\in [a,b]$. The Strichartz estimate \eqref{STR3} implies
$$
\| U(t-s) u(s) \|_{L^{5k/4}_{x} L^{5k/2}_{t}} \leq \|u(s)\|_{H^1}\leq K.
$$
Therefore, there exists and open interval $I_s$, containing $s$, such that
$$
\| U(t-s) u(s) \|_{L^{5k/4}_{x} L^{5k/2}_{I_s}} < \delta(K),
$$
where $\delta(K)$ is given by Proposition \ref{SDGT}. So,
$$
\| u \|_{L^{5k/4}_{x} L^{5k/2}_{I_s}} < 2 \delta(K).
$$
Since $ [a,b]\subset \bigcup\limits_{s\in [a,b]} I_s$ we can find $\{s_1,\cdots, s_l\}\subset [a,b]$ such that $ [a,b]\subset \bigcup\limits_{j=1}^{l} I_{s_j}$. Finally,
$$
\| u \|_{L^{5k/4}_{x} L^{5k/2}_{[a,b]}} \leq \sum_{j=1}^l \| u \|_{L^{5k/4}_{x} L^{5k/2}_{I_{s_j}}}\leq 2l\delta(K)<\infty.
$$

\end{remark}

Next result establishes a criterion for an $H^1$ solution to scatter. It says that as long as a global uniformly bounded solution satisfies $\|u\|_{L^{5k/4}_xL^{5k/2}_t}<\infty$ it scatters in both directions.

\begin{proposition}[$H^1$-scattering]\label{PROPSCAT} If $u_0\in H^1(\R)$, $u(t)$ is a global solution of the integral equation \eqref{IntEq}, with $t_0=0$, such that $\sup_{t\in \R}\|u(t)\|_{H^1}<\infty$ and $\|u\|_{L^{5k/4}_xL^{5k/2}_{[0,+\infty)}}<\infty$, then there exists $\phi^{+}\in H^1$ such that
\begin{equation}\label{SCATT}
\lim_{t\rightarrow +\infty} \|u(t)-U(t)\phi^+\|_{H^1}=0.
\end{equation}
Also, if $\|u\|_{L^{5k/4}_xL^{5k/2}_{(-\infty,0]}}<\infty$ then there exists $\phi^{-}\in H^1$ such that
\begin{equation*}
\lim_{t\rightarrow -\infty} \|u(t)-U(t)\phi^-\|_{H^1}=0.
\end{equation*}
\end{proposition}
\begin{proof}
The proof is quite standard, so we give only the main steps.
Let 
\begin{equation*}
\phi^+=u_0+\int_0^{\infty}U(-t')\partial_x(u^{k+1})(t')dt'.
\end{equation*}
Since $u$ is a solution of \eqref{IntEq} with $t_0=0$, we have 
\begin{equation*}
u(t)-U(t)\phi^+=-\int_t^{\infty}U(t-t')\partial_x(u^{k+1})(t')dt'.
\end{equation*}
Therefore,  from Lemma \ref{kpvlemma} (see also Remark \ref{infinity}), it is easy to see that
\begin{equation*}
\|u(t)-U(t)\phi^+\|_{H^1}\leq \|u\|^k_{L^{5k/4}_xL^{5k/2}_{[t,\infty)}}\left( \|u\|_{L^{5}_xL^{10}_{[0,+\infty)}}+ \|u_x\|_{L^{5}_xL^{10}_{[0,+\infty)}} \right). 
\end{equation*}

Since $\|u\|_{L^{5k/4}_xL^{5k/2}_{[t,\infty)}}\rightarrow 0$, as $t\rightarrow \infty$, to obtain \eqref{SCATT} it suffices to verify that 
\begin{equation}\label{SCATT2}
\|u\|_{L^{5}_xL^{10}_{[0,+\infty)}}+ \|u_x\|_{L^{5}_xL^{10}_{[0,+\infty)}} <\infty. 
\end{equation}
Indeed, since $\|u\|_{L^{5k/4}_xL^{5k/2}_{[0,+\infty)}}<\infty$, for any given $\delta>0$ we decompose $[0,\infty)$ into $N$ many intervals $I_j=[t_j,t_{j+1})$ such that $\|u\|_{L^{5k/4}_xL^{5k/2}_{I_j}}<\delta$ for all $j=1,\dots, N$. On the time interval $I_j$ we consider the integral equation
\begin{equation*}%\label{IEMOD}
u(t)=U(t-t_j)u(t_j)+\int_{t_j}^tU(t-t')\partial_x(u^{k+1})(t')dt'.
\end{equation*}

By using Lemma \ref{lemma2} with $(p_1,q_1,\alpha_1)=(5,10,0)$ and $(p_2,q_2,\alpha_2)=(\infty,2,1)$, we get
\begin{equation*}
\|u\|_{L^{5}_xL^{10}_{I_j}}+ \|u_x\|_{L^{5}_xL^{10}_{I_j}} \leq 2\left( \|u(t_j)\|_{H^1} + \delta^k (\|u\|_{L^{5}_xL^{10}_{I_j}}+ \|u_x\|_{L^{5}_xL^{10}_{I_j}})\right).
\end{equation*}
Taking $\delta>0$ sufficiently small such that $2\delta^k<1/2$, we have $\|u\|_{L^{5}_xL^{10}_{I_j}}+ \|u_x\|_{L^{5}_xL^{10}_{I_j}} \lesssim  \sup_{t\in \R}\|u(t)\|_{H^1}$. By summing over all the $N$ intervals we obtain \eqref{SCATT2}.

In a similar fashion we obtain the second statement, which finishes the
proof of Proposition \ref{PROPSCAT}.
\end{proof}

Next, we prove a perturbation result. We follow the exposition in \cite{KKSV} (see also \cite{HR08} and \cite{CFX}). We start with the following stability result, where we assume that the perturbed solution has small Strichartz norm.

\begin{proposition}[Perturbation theory I]\label{STPT}
For some error function $e$, let $\widetilde{u}$  be a global solution (in the sense of the appropriated integral equation) to
\begin{equation*}
\partial_t \widetilde{u}+\partial_{xxx}\widetilde{u}-\partial_x(\widetilde{u}^{k+1})=\partial_xe,
\end{equation*}
with initial data $\widetilde{u}_0\in H^1$, satisfying 
$$
\sup_{t \in \R}\|\widetilde{u}(t)\|_{H^1}\leq M \quad {\rm and} \quad \|\widetilde{u}\|_{_{L^{5k/4}_xL^{5k/2}_t}}\leq \varepsilon,
$$
for some positive constant $M$ and some small $0<\varepsilon<\varepsilon_0=\varepsilon_0(M,M')$.

Let $u_0\in H^1$ be such that 
$$\|u_0-\widetilde{u}_0\|_{H^1}\leq M' \quad {\rm and} \quad \|U(t)(u_0-\widetilde{u}_0)\|_{_{L^{5k/4}_xL^{5k/2}_t}}\leq \varepsilon,$$
for some positive constant $M'$.

Moreover, assume the following smallness condition on  $e$,
\begin{equation*}%\label{ERRORT}
\|D^{s_k}_xe\|_{L^1_xL^2_t}+\|D_xe\|_{L^1_xL^2_t}+\|e\|_{L^1_xL^2_t}\leq \varepsilon.
\end{equation*}

Then, there exists a global solution, say, $u$, to the generalized KdV equation \eqref{gkdv} with initial data $u_0$ at $t=0$ satisfying
\[
\begin{split}
	\|D^{s_k}_x(u^{k+1}-\widetilde{u}^{k+1})\|_{L^1_xL^2_t} + \|D_x(u^{k+1}-\widetilde{u}^{k+1})\|_{L^1_xL^2_t} +\|u^{k+1}-\widetilde{u}^{k+1}\|_{L^1_xL^2_t} \lesssim &\;\varepsilon,\\
	\|u\|_{L^{5k/4}_xL^{5k/2}_t}\lesssim & \; \varepsilon,\\
	\sup_{t \in \R}\|{u}(t)\|_{H^1}+\|D^{s_k}_xu\|_{L^5_xL^{10}_t} + \|D_xu\|_{L^5_xL^{10}_t} +\|u\|_{L^5_xL^{10}_t} \lesssim &\;  C(M,M').
\end{split}
\]
\end{proposition}

\begin{proof}
First we show that for $\varepsilon_0>0$ sufficiently small, if 
$$\|\widetilde{u}\|_{_{L^{5k/4}_xL^{5k/2}_t}}\leq \varepsilon_0$$
then
$$\|D^{s_k}_x\widetilde{u}\|_{_{L^{5}_xL^{10}_t}}\leq 2cM,$$
for some constant $c>0$. 
Indeed, consider the integral equation associated to the solution $\widetilde{u}$, that is,
\begin{equation}\label{IE}
\widetilde{u}(t)=U(t)\widetilde{u}_0+\int_0^tU(t-t')\partial_x(\widetilde{u}^{k+1}+e)(t')dt'.
\end{equation}
Now, applying Lemmas \ref{lemma2} to the integral equation \eqref{IE} and using Lemma \ref{lemma6}, we obtain
\begin{eqnarray*}
\|D^{s_k}_x\widetilde{u}\|_{L^5_xL^{10}_t}&\leq&
c\|D^{s_k}_xU(t)\widetilde{u}_0\|_{L^5_xL^{10}_t}+c\|\widetilde{u}\|_{L^{5k/4}_xL^{5k/2}_t}^k\|D^{s_k}_x\widetilde{u}\|_{L^5_xL^{10}_t} + \|D^{s_k}e\|_{L^1_xL^2_t}\\
&\leq& cM+c\varepsilon_0^k\|D^{s_k}_x\widetilde{u}\|_{L^5_xL^{10}_t}+\varepsilon_0.
\end{eqnarray*}
By taking $\varepsilon_0$ sufficiently small we conclude the claim.

Similar estimates also imply
$$\|D_x\widetilde{u}\|_{_{L^{5}_xL^{10}_t}}\leq 2cM \quad \textrm{ and } \quad 
\|\widetilde{u}\|_{_{L^{5}_xL^{10}_t}}\leq 2cM.$$

Since $u_0\in H^1$, we can assume that the global solution $u$ already exists. Thus, roughly speaking,  we need only to give the stated estimates for an appropriate $\varepsilon_0$. Let $w=u-\widetilde{u}$, then $w$ solves the equation
\begin{equation}\label{IEw}
\partial_t w+\partial_x^3 w-\partial_x((w+\widetilde{u})^{k+1}-\widetilde{u}^{k+1}-e)=0.
\end{equation}
We will proceed to see how the constant $\varepsilon_0$ affect the solution $w$.
Define 
$$
\Phi(w)(t):=U(t)w_0+\int_0^tU(t-t')\partial_x((w+\widetilde{u})^{k+1}-\widetilde{u}^{k+1}-e)(t')dt'
$$
and let $B_{a,b}$ be the set all functions on $\R^2$ such that $\|w\|_{L^{5k/4}_xL^{5k/2}_t}\leq a$ and
$$
\max\{\|D^{s_k}_xw\|_{L^5_xL^{10}_t}, \|D_xw\|_{L^5_xL^{10}_t}, \|w\|_{L^5_xL^{10}_t}, \|D_xw\|_{L^{\infty}_tL^2_x}, \|w\|_{L^{\infty}_tL^2_x}\}\leq b.
$$

Applying Lemma \ref{lemma1}, Lemma \ref{lemma12}, Lemma \ref{lemma2} and the Leibnitz rule for fractional derivative (Lemma \ref{lemma6}), we obtain
\begin{eqnarray*}
\|\Phi(w)\|_{L^{5k/4}_xL^{5k/2}_t}&\leq&
\|U(t)w_0\|_{L^{5k/4}_xL^{5k/2}_t}+c\left(\sum_{n=0}^k\|D^{s_k}_xw\|_{L^5_xL^{10}_t}\|w\|_{L^{5k/4}_xL^{5k/2}_t}^{k-n}
\|\widetilde{u}\|_{L^{5k/4}_xL^{5k/2}_t}^{n}\right.\\
&&+\left. \sum_{n=1}^k\|w\|_{L^{5k/4}_xL^{5k/2}_t}^{k-n+1}\|D^{s_k}_x\widetilde{u}\|_{L^5_xL^{10}_t}
\|\widetilde{u}\|_{L^{5k/4}_xL^{5k/2}_t}^{n-1} \right)\\
&\leq& \varepsilon +c\|D^{s_k}_xw\|_{L^5_xL^{10}_t}\|w\|_{L^{5k/4}_xL^{5k/2}_t}^{k}+
cM\|w\|_{L^{5k/4}_xL^{5k/2}_t}^{k}\\
&&+c\sum_{n=1}^k\varepsilon^{n}\|D^{s_k}_xw\|_{L^5_xL^{10}_t}\|w\|_{L^{5k/4}_xL^{5k/2}_t}^{k-n}+
cM\sum_{n=1}^{k-1}\varepsilon^{n}\|w\|_{L^{5k/4}_xL^{5k/2}_t}^{k-n}.
\end{eqnarray*}
By similar estimates, we also have
\begin{eqnarray*}
\|D^{s_k}_x\Phi(w)\|_{L^5_xL^{10}_t}&\leq& cM' +c\|D^{s_k}_xw\|_{L^5_xL^{10}_t}\|w\|_{L^{5k/4}_xL^{5k/2}_t}^{k}+
cM\|w\|_{L^{5k/4}_xL^{5k/2}_t}^{k}\\
&&+c\sum_{n=1}^k\varepsilon^{n}\|D^{s_k}_xw\|_{L^5_xL^{10}_t}\|w\|_{L^{5k/4}_xL^{5k/2}_t}^{k-n}+
cM\sum_{n=1}^{k-1}\varepsilon^{n}\|w\|_{L^{5k/4}_xL^{5k/2}_t}^{k-n},
\end{eqnarray*}
\begin{eqnarray*}
\|D_x\Phi(w)\|_{L^5_xL^{10}_t}&\leq& cM' +c\|D_xw\|_{L^5_xL^{10}_t}\|w\|_{L^{5k/4}_xL^{5k/2}_t}^{k}+
cM\|w\|_{L^{5k/4}_xL^{5k/2}_t}^{k}\\
&&+c\sum_{n=1}^k\varepsilon^{n}\|D_xw\|_{L^5_xL^{10}_t}\|w\|_{L^{5k/4}_xL^{5k/2}_t}^{k-n}+
cM\sum_{n=1}^{k-1}\varepsilon^{n}\|w\|_{L^{5k/4}_xL^{5k/2}_t}^{k-n},
\end{eqnarray*}
\begin{eqnarray*}
\|\Phi(w)\|_{L^5_xL^{10}_t}&\leq& cM' +c\|w\|_{L^5_xL^{10}_t}\|w\|_{L^{5k/4}_xL^{5k/2}_t}^{k}+
cM\|w\|_{L^{5k/4}_xL^{5k/2}_t}^{k}\\
&&+c\sum_{n=1}^k\varepsilon^{n}\|w\|_{L^5_xL^{10}_t}\|w\|_{L^{5k/4}_xL^{5k/2}_t}^{k-n}+
cM\sum_{n=1}^{k-1}\varepsilon^{n}\|w\|_{L^{5k/4}_xL^{5k/2}_t}^{k-n}.
\end{eqnarray*}
The estimates of the norms $\|D_x\Phi(w)\|_{L^{\infty}_tL^2_x}$ and $\|\Phi(w)\|_{L^{\infty}_tL^2_x}$ are exactly the same. First, we choose $b=2cM'$, $a$ such that $ca^k\leq cM'/4$ and $cMa^k\leq cM'/4$. Moreover, we can choose $\varepsilon_0$ sufficiently small so that
$$
c\sum_{n=1}^k \varepsilon^nba^{k-n}+cM\sum_{n=1}^{k-1} \varepsilon^na^{k-n}\leq cM'/2.
$$
Thus
$$\max\left\{\|D^{s_k}_x\Phi(w)\|_{L^5_xL^{10}_t}, \|D_x\Phi(w)\|_{L^5_xL^{10}_t}, \|\Phi(w)\|_{L^5_xL^{10}_t}, \|D_x\Phi(w)\|_{L^{\infty}_tL^2_x}, \|\Phi(w)\|_{L^{\infty}_tL^2_x}\right\}\leq b.$$
Next,  if one chooses $\varepsilon=a/2$ and $a$ so that
$$
a^k(cb+cM)\Big(1+\sum_{n=1}^{k}(1/2)^n+\sum_{n=1}^{k-1}(1/2)^n\Big)\leq a/2
$$ 
we also have
$$
\|\Phi(w)\|_{L^{5k/4}_xL^{5k/2}_t}\leq a.
$$

Therefore, standard arguments imply that
$$
\|w\|_{L^{5k/4}_xL^{5k/2}_t}\leq 2\varepsilon.
$$

Moreover, our choice of parameters $a$ and $b$ also imply
$$\|D^{s_k}_x(u^{k+1}-\widetilde{u}^{k+1})\|_{L^1_xL^2_t} + \|D_x(u^{k+1}-\widetilde{u}^{k+1})\|_{L^1_xL^2_t}+\|u^{k+1}-\widetilde{u}^{k+1}\|_{L^1_xL^2_t} \lesssim \varepsilon,$$
which concludes the proof.
\end{proof}

\begin{remark}%\label{remsmall}
If $\widetilde{u}$ is not a global solution but instead is a local one defined on $\R\times I$ with initial data $\widetilde{u}(t_0)=\widetilde{u}_0$, where $I\subset\R$ is some time interval, and the assumption on the norms $L^p_xL^q_t$ are replaced by the norms in $L^p_xL^q_I$, we can repeat the proof of Proposition \ref{STPT} to obtain a solution $u$ defined on $\R\times I$ and satisfying the same conclusions but with the norms $L^p_xL^q_t$ replaced by $L^p_xL^q_I$.
\end{remark}

In the next proposition, we show that Proposition \ref{STPT} can also be iterated to obtain a stability result without assuming smallness in the Strichartz norm $L^{5k/4}_xL^{5k/2}_t$.

\begin{proposition}[Perturbation theory II]\label{LTPT} For some error function $e$, let $\widetilde{u}$  be a global solution (in the sense of the appropriated integral equation) to
\begin{equation*}
\partial_t \widetilde{u}+\partial_{xxx}\widetilde{u}-\partial_x(\widetilde{u}^{k+1})=\partial_xe,
\end{equation*}
with initial data $\widetilde{u}_0\in H^1$ at $t=0$, satisfying 
$$\sup_{t \in \R}\|\widetilde{u}(t)\|_{H^1}\leq M \quad {\rm and} \quad \|\widetilde{u}\|_{_{L^{5k/4}_xL^{5k/2}_t}}\leq L,$$
for some positive constants $M,L$.

Let $u_0\in H^1$ be such that 
$$\|u_0-\widetilde{u}_0\|_{H^1}\leq M' \quad {\rm and} \quad \|U(t)(u_0-\widetilde{u}_0)\|_{_{L^{5k/4}_xL^{5k/2}_t}}\leq \varepsilon,$$
for some positive constant $M'$ and some small $0<\varepsilon<\varepsilon_1=\varepsilon_1(M,M',L)$.

Moreover, assume also the following smallness condition on the error $e$ 
$$\|D^{s_k}_xe\|_{L^1_xL^2_t}+\|D_xe\|_{L^1_xL^2_t}+\|e\|_{L^1_xL^2_t}\leq \varepsilon.$$

Then, there exists a global solution, say, $u$, to the generalized KdV equation \eqref{gkdv} with initial data $u_0$ at $t=0$ satisfying
\begin{eqnarray*}
\|u-\widetilde{u}\|_{L^{5k/4}_xL^{5k/2}_t}&\leq& C(M,M')\varepsilon,\\
\sup_{t\in \R}\|(u-\widetilde{u})(t)\|_{H^{1}_x}  &\leq& C(M,M'),\\
\|u\|_{L^{5k/4}_xL^{5k/2}_t} + \|D^{s_k}_xu\|_{L^5_xL^{10}_t} + \|D_xu\|_{L^5_xL^{10}_t} +\|u\|_{L^5_xL^{10}_t} &\leq& C(M,M').
\end{eqnarray*}
\end{proposition}
\begin{proof}
Let $\varepsilon_0$ be given in Proposition \ref{STPT}. Since $\|\widetilde{u}\|_{_{L^{5k/4}_xL^{5k/2}_t}}\leq L$, for any $0<\varepsilon < \varepsilon_0$ to be determined later,  we can split the interval $I=[0,\infty)$ into $N=N(L,\varepsilon)$ intervals $I_j=[t_j,t_{j+1})$ so that $\|\widetilde{u}\|_{_{L^{5k/4}_xL^{5k/2}_{I_j}}}\leq \varepsilon$. In this case, we already know from the proof of Proposition \ref{STPT} that
$$
\|D^{s_k}_x\widetilde{u}\|_{_{L^{5}_xL^{10}_{I_j}}}\leq 2cM, \quad \|D_x\widetilde{u}\|_{_{L^{5}_xL^{10}_{I_j}}}\leq 2cM \quad \textrm{ and } \quad 
\|\widetilde{u}\|_{_{L^{5}_xL^{10}_{I_j}}}\leq 2cM.
$$

 Let $u=\widetilde{u}+w$, then $w$ solves equation \eqref{IEw} in the interval $I_j$. The integral equation in this case reads as 
\begin{equation}\label{IEw2}
w(t)=U(t-t_j)w(t_j)+\int_{t_j}^tU(t-t')\partial_x((w+\widetilde{u})^{k+1}-\widetilde{u}^{k+1}-e)(t')dt'.
\end{equation}

Thus, for $0<\varepsilon < \varepsilon_0$, assuming a priori that $\varepsilon_1\leq \varepsilon_0$, Proposition \ref{STPT} and our assumptions imply
 $$
\|w\|_{L^{5k/4}_xL^{5k/2}_{I_j}}\leq 2\varepsilon,
$$
$$
\|w\|_{L^{\infty}_{I_j}H^1} + \|D^{s_k}_xw\|_{L^5_xL^{10}_{I_j}} + \|D_xw\|_{L^5_xL^{10}_{I_j}} +\|w\|_{L^5_xL^{10}_{I_j}} \leq C(M,M'),
$$
and
$$\|D^{s_k}_x(u^{k+1}-\widetilde{u}^{k+1})\|_{L^1_xL^2_{I_j}} + \|D_x(u^{k+1}-\widetilde{u}^{k+1})\|_{L^1_xL^2_{I_j}} +\|u^{k+1}-\widetilde{u}^{k+1}\|_{L^1_xL^2_{I_j}} \leq C(M,M') \varepsilon.$$

By choosing $\varepsilon_1$ sufficiently small depending on $N,M,M'$, we can apply this procedure inductively to obtain, for each $0\leq j<N$ and all $0<\varepsilon < \varepsilon_1$,
$$
\|u-\widetilde{u}\|_{L^{5k/4}_xL^{5k/2}_{I_j}}\leq C(j,M,M')\varepsilon,
$$
$$
\|D^{s_k}_x(u^{k+1}-\widetilde{u}^{k+1})\|_{L^1_xL^2_{I_j}} + \|D_x(u^{k+1}-\widetilde{u}^{k+1})\|_{L^1_xL^2_{I_j}} +\|u^{k+1}-\widetilde{u}^{k+1}\|_{L^1_xL^2_{I_j}}  \leq C(j,M,M') \varepsilon,
$$
and
$$
\|u\|_{L^{5k/4}_xL^{5k/2}_{I_j}} + \|D^{s_k}_xu\|_{L^5_xL^{10}_{I_j}} + \|D_xu\|_{L^5_xL^{10}_t} +\|u\|_{L^5_xL^{10}_{I_j}} \leq C(j,M,M')
$$
provided we can show that 
\begin{equation}\label{LTPTeq1}
\|u(t_j)-\widetilde{u}(t_j)\|_{H^1}\leq 2M'
\end{equation}
and 
\begin{equation}\label{LTPTeq2}
\|U(t-t_j)(u(t_j)-\widetilde{u}(t_j))\|_{_{L^{5k/4}_xL^{5k/2}_{I_j}}}\leq C(j,M,M')\varepsilon \leq \varepsilon_0.
\end{equation}

First we prove  \eqref{LTPTeq1}. Indeed, applying Lemma \ref{lemma2} and the inductive hypotheses, we obtain
\begin{eqnarray*}
\|u(t_j)-\widetilde{u}(t_j)\|_{H^1}&\leq& \|u_0-\widetilde{u}_0\|_{H^1} +\|D_x(u^{k+1}-\widetilde{u}^{k+1})\|_{L^1_xL^2_{[0,t_j]}} \\
&&+\|u^{k+1}-\widetilde{u}^{k+1}\|_{L^1_xL^2_{[0,t_j]}} + \|D_xe\|_{L^1_xL^2_t}+\|e\|_{L^1_xL^2_t}\\
&\leq& M' +2\sum_{k=1}^{j}C(k,M,M')\varepsilon + 2\varepsilon.
\end{eqnarray*}
Taking $\varepsilon_1=\varepsilon_1(N,M,M')$ sufficiently small we obtain the desired inequality.

Next, we turn our attention to \eqref{LTPTeq2}. Using the integral equation \eqref{IEw2} with $t_j$ replaced by $t_{j+1}$, we can easily conclude
$$
U(t-t_{j+1})w(t_{j+1})=U(t-t_j)w(t_j)+\int_{t_j}^{t_{j+1}}U(t-t')\partial_x((w+\widetilde{u})^{k+1}-\widetilde{u}^{k+1}-e)(t')dt'.
$$

Therefore, from Lemma \ref{lemma2} and the inductive hypotheses, we deduce
\begin{eqnarray*}
\|U(t-t_j)(u(t_j)-\widetilde{u}(t_j))\|_{_{L^{5k/4}_xL^{5k/2}_{I_j}}}&\leq&\|U(t)(u_0-\widetilde{u}_0)\|_{_{L^{5k/4}_xL^{5k/2}_t}} + \|D_x^{s_k}e\|_{L^1_xL^2_t} \\
&&+\|D_x^{s_k}(u^{k+1}-\widetilde{u}^{k+1})\|_{L^1_xL^2_{[0,t_j]}}\\
&\leq& 2\varepsilon + \sum_{k=1}^{j}C(k,M,M')\varepsilon,
\end{eqnarray*}
and again taking $\varepsilon_1=\varepsilon_1(N,M,M')$ sufficiently small we prove \eqref{LTPTeq2}. The same analysis can be performed on the interval $(-\infty,0]$. The proof of the proposition is thus completed.
\end{proof}

We finish this section by recalling the notion of a nonlinear profile
\begin{definition} \label{NLP}
Let $\psi \in {H}^{1}$ and $\{ t_{n} \}_{n \in \mathbb{N}}$ be a sequence with $\underset{n \to \infty}{\lim} t_{n} = \overline{t} \in [-\infty, \infty]$. We say that $u(x,t)$ is a nonlinear profile associated with $(\psi, \{ t_{n} \}_{n \in  \mathbb{N}})$ if there exists an interval $I = (a,b)$ with $\overline{t} \in I$ (if 
$\,\overline{t} = \pm \infty$, then $I = (a, +\infty)$ or $I = (-\infty, b)$, as appropriate) such that $u$ solves \eqref{gkdv} in $I$ and 
\begin{equation*}
\lim_{n \to \infty} \| u(t_{n}) - U(t_{n}) \psi \|_{{H}^{1}} = 0.
\end{equation*}
\end{definition}

\begin{remark}\label{4.2}
There always exists a unique nonlinear profile associated with $(\psi, \{ t_{n} \}_{n \in  \mathbb{N}})$.
For a proof of this fact, see the analogous one in Proposition 4.1 and Remark 4.2 in Farah and Pigott \cite{FP16}. Therefore, we can also define the maximal interval $\bar{I}$ of existence for the nonlinear profile associated with $(\psi, \{ t_{n} \}_{n \in \mathbb{N}})$. Moreover, if $k$ is even in \eqref{gkdv}, then $\bar{I}=\R$ in view of Theorem \ref{global5}.
\end{remark}

\section{Profile Decomposition}\label{sec4}

In this section we establish that any bounded sequence in $H^1$ has a  decomposition into linear profiles and a reminder with a suitable asymptotic smallness property in an adequate norm. Our strategy follows the ones, for instance, in  \cite{dhr} and \cite{CFX}.

We start  with the following lemma (see also Fang \textit{et al.} \cite[Lemma 5.3]{CFX}).

\begin{lemma}\label{LemmaCaz}
Let $\{z_n\}_{n\in \mathbb{N}}\subseteq H^1$ and $0\neq \psi\in H^1$ satisfying
$$
z_n \rightharpoonup 0 \peq \textrm{ and } \peq U(t_n)z_n(\cdot + x_n) \rightharpoonup \psi  \peq \textrm{in} \peq H^1, \peq \textrm{as} \peq n\rightarrow \infty.
$$
Then
$$
|t_n|+|x_n|\rightarrow \infty , \peq \textrm{as} \peq n\rightarrow \infty.
$$
\end{lemma}

\begin{proof}
Suppose by contradiction that there exists a subsequence,  still indexed by $n$, and a positive number $M$ satisfying
$$
|t_n|+|x_n|\leq M , \peq \textrm{for all} \peq n\in \mathbb{N}.
$$
Therefore, without loss of generality, there exist $\bar{t}, \bar{x}\in \R$ such that
$$
t_n\rightarrow \bar{t} \peq \textrm{ and } \peq x_n\rightarrow \bar{x}. 
$$

We claim that $U(t_n)z_n(\cdot + x_n) \rightharpoonup 0$ in $H^1$, which is an absurd. In fact, by standard density arguments, we only need to show that
$$
\left(U(t_n)z_n(\cdot + x_n) , \zeta \right)_{H^1} \rightarrow 0, \peq \textrm{as} \peq n\rightarrow \infty
$$
for all functions $\zeta$ in the Schwartz class. To see this, we write
\[
\begin{split}
	\left(U(t_n)z_n(\cdot + x_n) , \zeta \right)_{H^1} &= \left(z_n ,U(-t_n) \zeta(\cdot - x_n) \right)_{H^1} \\
	&= \left(z_n ,U(-\bar{t}) \zeta(\cdot - \bar{x}) \right)_{H^1}
	+\left(z_n ,U(-t_n) \zeta(\cdot - x_n)-U(-\bar{t}) \zeta(\cdot - \bar{x}) \right)_{H^1}\\
	&\leq \left(z_n ,U(-\bar{t}) \zeta(\cdot - \bar{x}) \right)_{H^1}\\
	&\quad
	+\|z_n\|_{H^1}\|U(-t_n) \zeta(\cdot - x_n)-U(-\bar{t}) \zeta(\cdot - \bar{x} )\|_{H^1}\to0, \;\;\mbox{as}\;n\to\infty.
\end{split}
\]
The first term on the r.h.s of the above inequality goes to zero because $z_n \rightharpoonup 0$. In addition, since $\{z_n\}_{n\in \mathbb{N}}$ is bounded in $H^1$, $t_n\rightarrow \bar{t}$ and $x_n\rightarrow \bar{x}$, by Lebesgue's Dominated Convergence Theorem, the norm involving the linear evolution also goes to zero. This proves the lemma.
\end{proof}

Next, we prove the following profile decomposition result.

\begin{theorem}\label{profdec}
 Let  $\{\phi_n\}_{n\in \mathbb{N}}$ be a bounded sequence in $H^1$. There exists a subsequence, which we still denote by $\{\phi_n\}_{n\in \mathbb{N}}$, and sequences $\{\psi^j\}_{j\in \mathbb{N}}\subset H^1$,  $\{W^j_n\}_{j, n \in \mathbb{N}}\subset H^1$, $\{t^j_n\}_{j, n \in \mathbb{N}}\subset \R$, $\{\bar{t}^j\}_{j\in \mathbb{N}}\subset [-\infty,+\infty]$ and $\{x^j_n\}_{j, n \in \mathbb{N}}\subset \R$, such that for every $l\geq1$,
\begin{equation*}
\phi_n=\sum_{j=1}^{l}U(t_n^j)\psi^j(\cdot-x^j_n) + W_n^l
\end{equation*}
and
\begin{equation*}
t_n^l\rightarrow \bar{t}^l,  \peqq \textrm{as} \peqq n\rightarrow \infty,
\end{equation*}
\begin{equation}\label{PYTHA}
\|\phi_n\|^2_{\dot{H}^{\lambda}}-\sum_{j=1}^{l}\|\psi^j\|^2_{\dot{H}^{\lambda}} - \|W_n^l\|^2_{\dot{H}^{\lambda}} \rightarrow 0,  \peqq \textrm{as} \peqq n\rightarrow \infty, \peqq \textrm{for all} \peqq 0\leq \lambda\leq 1,
\end{equation}
Furthermore, the time and space sequence have a pairwise divergence property: for $1\leq i \neq j\leq l$, we have
\begin{equation}\label{XT}
\lim_{n\rightarrow \infty} |t_n^i-t_n^j|+|x_n^i-x_n^j|=\infty.
\end{equation}
Finally, the reminder sequence has the following asymptotic smallness property
\begin{equation}\label{SPWNL}
\limsup_{n\rightarrow \infty} \|U(t)W^l_n\|_{L^{5k/4}_xL^{5k/2}_t} \rightarrow 0,  \peqq \textrm{as} \peqq l\rightarrow \infty.
\end{equation}
\end{theorem}
\begin{proof}
First of all, we observe if $\{W^l_n\}_{l, n \in \mathbb{N}}$ is a bounded sequence in $H^1$, by analytic interpolation and inequalities \eqref{reg1} and \eqref{STR2}, we have
\begin{equation}\label{vanisstr}
\begin{split}
\|U(t)W^l_n&\|_{L^{5k/4}_xL^{5k/2}_t}\\ & \lesssim \|D_x^{-1/k}U(t)W^l_n\|_{L^{k}_xL^{\infty}_t}^{2/5} \|D_x^{2/3k}U(t)W^l_n\|_{L^{3k/2}_{x,t}}^{3/5}\\
&\lesssim \|D_x^{-1/k}U(t)W^l_n\|_{L^{k}_xL^{\infty}_t}^{2/5} \|D_x^{1/q}U(t)W^l_n\|_{L^{q}_tL^{k+2}_x}^{2/5}  \|D_x^{1/(k+2)}U(t)W^l_n\|_{L^{k+2}_tL^{k(k+2)/4}_x}^{1/5}\\
&\lesssim \|D_x^{1/q}U(t)W^l_n\|_{L^{q}_tL^{k+2}_x}^{2/5}  \|D_x^{s_k}W^l_n\|_{L^2_x}^{3/5}\\
&\lesssim \|D_x^{1/q}U(t)W^l_n\|_{L^{q}_tL^{k+2}_x}^{2/5},
\end{split}
\end{equation}
where $\frac{2}{q}+\frac{1}{k+2}=\frac{2}{k}$.
In addition, by complex interpolation and Lemma \ref{lemma1}, we have
\begin{equation}\label{vanisstr1}
\begin{split}
 \|D_x^{1/q}U(t)W^l_n\|_{L^{q}_tL^{k+2}_x}&\lesssim \|D_x^{1/q'}U(t)W^l_n\|_{L^{q'}_tL^{k+2}_x}^{\theta} \|U(t)W^l_n\|_{L^{\infty}_tL^{k+2}_x}^{1-\theta}\\
 &\lesssim \|W^l_n\|_{L^{2}_x}^{\theta} \|U(t)W^l_n\|_{L^{\infty}_tL^{k+2}_x}^{1-\theta}\\
 &\lesssim \|U(t)W^l_n\|_{L^{\infty}_tL^{k+2}_x}^{1-\theta},
 \end{split}
 \end{equation}
where $(q',k+2)$ is a $L^2$-admissible Strichartz pair, that is,
$$
\dfrac{2}{q'}+\dfrac{1}{k+2}=\dfrac{1}{2} \quad \textrm{ and } \quad \theta=\dfrac{q'}{q}\in (0,1),\\
$$
So, in view of \eqref{vanisstr} and \eqref{vanisstr1}, for the reminder it will be suffice to show that
\begin{equation}\label{UWNL}
\limsup_{n\rightarrow \infty} \|U(t)W^l_n\|_{L^{\infty}_tL^{k+2}_x} \rightarrow 0,  \peqq \textrm{as} \peqq l\rightarrow \infty.
\end{equation}

Let $\zeta \in C^{\infty}_0(\R)$ such that
\begin{equation*}
\zeta(\xi)=\begin{cases}
1\;, \;\; |\xi|\leq1, \\
0\;, \;\; |\xi|\geq 2.
\end{cases}
\end{equation*}
Given any $\gamma>0$ define the function $\chi_{\gamma}$ through its Fourier transform by $\widehat{\chi}_{\gamma}(\xi)=\zeta(\xi/\gamma)$. Therefore, since $|\widehat{\chi}_{\gamma}(\xi)|\leq 1$, for any $u\in H^1$ and any $0\leq \lambda\leq1$, we obtain
\begin{equation*}%\label{gamma1}
\|\chi_{\gamma}\ast u\|_{\dot{H}^{\lambda}}\leq \|u\|_{\dot{H}^{\lambda}}.
\end{equation*}

We also have
\begin{equation}\label{gamma2}
\|u-\chi_{\gamma}\ast u\|_{\dot{H}^{\lambda}}\leq \gamma^{-(1-\lambda)}\|u\|_{\dot{H}^{1}},
\end{equation}
because
\begin{eqnarray*}
\|u-\chi_{\gamma}\ast u\|^2_{\dot{H}^{\lambda}} &=& \int_{\R}|\xi|^{2\lambda}(1-\widehat{\chi}_{\gamma}(\xi))^2|\widehat{u}(\xi)|^2d \xi\\
&\leq&  \int_{|\xi|\geq \gamma}|\xi|^{2(\lambda-1)}|\xi|^2|\widehat{u}(\xi)|^2d \xi\\
&\leq&  \gamma^{-2(1-\lambda)}\|u\|^2_{\dot{H}^{1}}.
\end{eqnarray*}
Moreover, by Plancherel's identity, we conclude
\begin{equation*}%\label{gamma31}
\begin{split}
|\chi_{\gamma}\ast u(0)|=&\left| \int_{\R} \widehat{\chi}_{\gamma}(\xi) \widehat{u}(\xi)d \xi \right|\leq \left| \int_{|\xi|\leq 2\gamma} \widehat{u}(\xi)d \xi \right|\\
\leq& \|u\|_{\dot{H}^{\lambda}}\left(\int_{|\xi|<2\gamma}|\xi|^{-2\lambda}d \xi\right)^{1/2}\\
\leq& \kappa \gamma^{\frac{1-2\lambda}{2}}\|u\|_{\dot{H}^{\lambda}},
\end{split}
\end{equation*}
for some constant $\kappa>0$.
In particular, for $\lambda=\dfrac{k}{2(k+2)}$,
\begin{equation}\label{gamma3}
|\chi_{\gamma}\ast u(0)|\leq \kappa \gamma^{\frac{1}{k+2}}\|u\|_{\dot{H}^{\frac{k}{2(k+2)}}}.
\end{equation}

On the other hand, in view of Sobolev's embedding, we also have
\begin{equation}\label{gamma4}
\|u\|_{L^{k+2}_x}\leq \beta \|u\|_{\dot{H}^{\frac{k}{2(k+2)}}},
\end{equation}
where the numbers $\kappa$ and $\beta$ are independent of $\gamma$. 

Let $A_1:= \limsup_{n\rightarrow \infty} \|U(t)\phi_n\|_{L^{\infty}_tL^{k+2}_x}$. If $A_1=0$ we may take for all $j,l\geq1$, $\psi^j=0$, $t_n^j=x_n^j=0$, $W_n^l=\phi_n$ and the proof is complete. Suppose now $A_1>0$, without loss of generality we may assume $\|U(t)\phi_n\|_{L^{\infty}_tL^{k+2}_x}\rightarrow A_1$, as $n\rightarrow \infty$. We claim that there exist sequences $\{t^1_n\}_{n \in \mathbb{N}}\subset \R$, $\{x^1_n\}_{n \in \mathbb{N}}\subset \R$ and a function $\psi^1\in H^1$, such that
\begin{equation}\label{Wphi1}
U(-t_n^1)\phi_n(\cdot + x_n^1) \rightharpoonup \psi^1  \peq \textrm{in} \peq H^1, \peq \textrm{as} \peq n\rightarrow \infty
\end{equation}
and
\begin{equation}\label{Wphi2}
\|\psi^1\|_{H^{1}}\geq c(\beta)C_1^{-\frac2k-\frac{2}{k+4}}A_1^{\frac{k+2}{k}+\frac{2}{k+4}},
\end{equation}
where $c(\beta)>0$ is a constant depending only on $\beta$ and $C_1:= \limsup_{n\rightarrow \infty} \|\phi_n\|_{H^{1}}$.

Indeed, since $U(\cdot)$ is an isometry on $\dot{H}^{\frac{k}{2(k+2)}}$, by choosing 
\begin{equation}\label{gamma}
\gamma=\left( \dfrac{4\beta C_1}{A_1}\right)^{\frac{2(k+2)}{k+4}},
\end{equation}
we deduce from \eqref{gamma2} and \eqref{gamma4} that, for $n$ large,
\begin{eqnarray*}
\|U(t)\phi_n-U(t)(\chi_{\gamma}\ast \phi_n)\|_{L^{\infty}_tL^{k+2}_x}&\leq& \beta \|U(t)\phi_n-U(t)(\chi_{\gamma}\ast \phi_n)\|_{L^{\infty}_t\dot{H}^{\frac{k}{2(k+2)}}}\\
&\leq& 2 \beta \gamma^{-(1-\frac{k}{2(k+2)})} C_1 \leq A_1/2.
\end{eqnarray*}
Thus,  for $n$ large,
$$
\|U(t)(\chi_{\gamma}\ast \phi_n)\|_{L^{\infty}_tL^{k+2}_x}\geq A_1/4.
$$

By complex interpolation and Strichartz estimates (Lemma \ref{lemma1}), we conclude
\begin{eqnarray*}
\|U(t)(\chi_{\gamma}\ast \phi_n)\|_{L^{\infty}_tL^{k+2}_x}&\leq& \|U(t)(\chi_{\gamma}\ast \phi_n)\|_{L^{\infty}_tL^{2}_x}^{2/(k+2)}
\|U(t)(\chi_{\gamma}\ast \phi_n)\|_{L^{\infty}_{x,t}}^{k/(k+2)}\\
&\leq& \|\phi_n\|_{L^{2}_x}^{2/(k+2)}\|U(t)(\chi_{\gamma}\ast \phi_n)\|_{L^{\infty}_{x,t}}^{k/(k+2)}\\
&\leq& (2C_1)^{2/(k+2)}\|U(t)(\chi_{\gamma}\ast \phi_n)\|_{L^{\infty}_{x,t}}^{k/(k+2)}.
\end{eqnarray*}
Combining these last two inequalities, we obtain for $n$ large
$$
\|U(t)(\chi_{\gamma}\ast \phi_n)\|_{L^{\infty}_{x,t}}\geq \left( \dfrac{A_1}{4}\right)^{\frac{k+2}{k}}\left( 2C_1\right)^{-\frac{2}{k}},
$$
from which it follows that there exist sequences $\{t^1_n\}_{n \in \mathbb{N}}\subset \R$ and  $\{x^1_n\}_{n \in \mathbb{N}}\subset \R$ such that
\begin{equation}\label{Wphi3}
|U(t_n^1)(\chi_{\gamma}\ast \phi_n)(x_n^1)| =\left| \int U(t_n^1)\phi_n(x_n^1-y)\chi_{\gamma}(y) dy\right|\geq \dfrac{1}{2}\left( \dfrac{A_1}{4}\right)^{\frac{k+2}{k}}\left( 2C_1\right)^{-\frac{2}{k}}.
\end{equation}
Now define  $\omega_n(\cdot) = U(t_n^1)\phi_n(\cdot +x_n^1)$. It is clear that $\|\phi_n\|_{H^{1}}=\|\omega_n\|_{H^{1}}$, therefore $\{\omega_n\}_{n \in \mathbb{N}}$ is uniformly bounded in $H^1$. Passing to a subsequence if necessary, there exists a function $\psi^1$ such that $\omega_n  \rightharpoonup \psi^1$ in $H^1$. Moreover, by \eqref{Wphi3} and \eqref{gamma3},
\begin{eqnarray*}
\dfrac{1}{2}\left( \dfrac{A_1}{4}\right)^{\frac{k+2}{k}}\left( 2C_1\right)^{-\frac{2}{k}}\leq \left| \int \psi^1(y)\chi_{\gamma}(-y) dy\right|=
|\chi_{\gamma}\ast \psi^1(0)|\leq  \kappa \gamma^{\frac{1}{k+2}}\|\psi^1\|_{{H}^{1}}.
\end{eqnarray*}
By using our choice of $\gamma$ in \eqref{gamma}, we conclude inequality \eqref{Wphi2}. 

Now, let $W_n^1=\phi_n-U(-t_n^1)\psi^1(\cdot -x_n^1)$. Given any $0\leq \lambda \leq 1$, it follows that
\begin{equation}\label{WN1}
\begin{split}
 \|W_n^1\|^2_{\dot{H}^{\lambda}}&= \|\phi_n\|^2_{\dot{H}^{\lambda}}+ \|U(-t_n^1)\psi^1(\cdot -x_n^1)\|^2_{\dot{H}^{\lambda}}-2\Re(\phi_n,U(-t_n^1)\psi^1(\cdot -x_n^1))_{\dot{H}^{\lambda}} \\
 &= \|\phi_n\|^2_{\dot{H}^{\lambda}}+ \|\psi^1\|^2_{\dot{H}^{\lambda}}-2\Re(\omega_n,\psi^1)_{\dot{H}^{\lambda}}.
 \end{split}
\end{equation}
Taking the limit, as $n\to\infty$, we obtain
$$
\limsup_{n\rightarrow \infty} \|W^1_n\|_{\dot{H}^{\lambda}}=\limsup_{n\rightarrow \infty} \|\phi_n\|_{\dot{H}^{\lambda}}-\|\psi^1\|^2_{\dot{H}^{\lambda}}.
$$
The above limit also implies 
$$C_2:= \limsup_{n\rightarrow \infty} \|W^1_n\|_{H^{1}}\leq C_1= \limsup_{n\rightarrow \infty} \|\phi_n\|_{H^{1}}.$$

Let $A_2:= \limsup_{n\rightarrow \infty} \|U(t)W^1_n\|_{L^{\infty}_tL^{k+2}_x}$. If $A_2=0$, there is nothing to prove. Again, the only case we need to consider is $A_2>0$. Repeating the above procedure, with $\phi_n$ replaced by $W^1_n$ we can find sequences $\{t^2_n\}_{n \in \mathbb{N}}\subset \R$, $\{x^2_n\}_{n \in \mathbb{N}}\subset \R$ and a function $\psi^2\in H^1$, such that
\begin{equation*}%\label{Wphi12}
U(t_n^2)W^1_n(\cdot + x_n^2) \rightharpoonup \psi^2  \peq \textrm{in} \peq H^1, \peq \textrm{as} \peq n\rightarrow \infty,
\end{equation*}
and
\begin{equation*}%\label{Wphi22}
\|\psi^2\|_{H^{1}}\geq c(\beta)C_2^{-\frac{2}{k}-\frac{2}{k+4}}A_2^{\frac{k+2}{k}+\frac{2}{k+4}}.
\end{equation*}
Let $z_n=U(t_n^1)W^1_n(\cdot + x_n^1)$, therefore by \eqref{Wphi1}, $z_n \rightharpoonup 0$ in $H^1$. Moreover,
$$
U(t^2_n-t^1_n)z_n(\cdot+ x_n^2-x^1_n)=U(t_n^2)W^1_n(\cdot + x_n^2) \rightharpoonup \psi^2  \peq \textrm{in} \peq H^1,
$$
by construction. Thus, from Lemma \ref{LemmaCaz}, we conclude that
\begin{equation*}
\lim_{n\rightarrow \infty} |t_n^1-t_n^2|+|x_n^1-x_n^2|=\infty.
\end{equation*}

Set $W_n^2=W_n^1-U(-t_n^2) \psi^2(\cdot-x_n^2)$. By the same argument as the one used in \eqref{WN1} we have for all $0\leq \lambda \leq 1$
\begin{eqnarray*}
 \limsup_{n\rightarrow \infty} \|W_n^2\|^2_{\dot{H}^{\lambda}} &=& \limsup_{n\rightarrow \infty}\|W_n^1\|^2_{\dot{H}^{\lambda}}- \|\psi^2\|^2_{\dot{H}^{\lambda}}\\
 &=&\limsup_{n\rightarrow \infty}\|\phi_n\|^2_{\dot{H}^{\lambda}}- \|\psi^1\|^2_{\dot{H}^{\lambda}} - \|\psi^2\|^2_{\dot{H}^{\lambda}}.
\end{eqnarray*}
Therefore, relation \eqref{PYTHA} holds in this case, which again implies
$$C_3:= \limsup_{n\rightarrow \infty} \|W^2_n\|_{H^{1}}\leq C_1= \limsup_{n\rightarrow \infty} \|\phi_n\|_{H^{1}}.$$

Next, we construct the functions $\psi^j$, $j> 2$ inductively, applying the procedure described above to the sequences $\{W^{j-1}_n\}_{n\in \mathbb{N}}$. Let $l>2$.  By assuming that $\psi^j$, $x_n^j$, $t^j_n$ and $W_n^j$ are known for $j\in \{1,\dots,l-1\}$, we consider
$$
A_l:= \limsup_{n\rightarrow \infty} \|U(t)W^{l-1}_n\|_{L^{\infty}_tL^{k+2}_x}.
$$ 
If $A_l=0$ we are done. Assume $A_l>0$ and apply the above procedure to the sequence $\{W_n^{l-1}\}_{n\in \mathbb{N}}$ to obtain, passing to a subsequence if necessary, sequences  $\{t^l_n\}_{n \in \mathbb{N}}\subset \R$, $\{x^l_n\}_{n \in \mathbb{N}}\subset \R$ and a function $0\neq \psi^l\in H^1$ (if $\psi^l=0$ the result is trivial), such that
\begin{equation}\label{Wphi1l}
U(t_n^l)W^{l-1}_n(\cdot + x_n^l) \rightharpoonup \psi^l  \peq \textrm{in} \peq H^1, \peq \textrm{as} \peq n\rightarrow \infty
\end{equation}
and
\begin{equation*}%\label{Wphi2l}
\|\psi^l\|_{H^{1}}\geq c(\beta)C_l^{-\frac{2}{k}-\frac{2}{k+4}}A_l^{\frac{k+2}{k}+\frac{2}{k+4}},
\end{equation*}
where $C_l:= \limsup_{n\rightarrow \infty} \|W^{l-1}_n\|_{H^{1}}$.

Next we prove \eqref{PYTHA} and \eqref{XT} by induction. First, assume that \eqref{PYTHA} and \eqref{XT} holds for $j, k \in \{1,\dots,l-1\}$. Let $j \in \{1,\dots,l-1\}$. By definition,
$$
W_n^{l-1}=W_n^{j-1} - U(-t_n^j)\psi^j(\cdot-x_n^j)-U(-t_n^{j+1})\psi^{j+1}(\cdot-x_n^{j+1})-\cdots - U(-t_n^{l-1})\psi^{l-1}(\cdot-x_n^{l-1}).
$$
Therefore, applying the operator $U(t_n^j)$ and the shift $x_n^j$, we obtain
\[
\begin{split}
U(t_n^j)W_n^{l-1}(\cdot+x_n^j)= U(t_n^j)W_n^{j-1}(\cdot+x_n^j) - \psi^j 
-\sum_{k=j+1}^{l-1} U(t_n^j-t_n^k)\psi^k(\cdot+x_n^j-x_n^k).
\end{split}
\]
The first term on the r.h.s goes to zero, weakly in $H^1$, by definition and the same happens with the sum by our induction hypotheses. Thus, $U(-t_n^j)W_n^{l-1}(\cdot+x_n^j) \rightharpoonup 0$ in $H^1$, as $n\rightarrow \infty$. By applying Lemma \ref{LemmaCaz}, we deduce that
\begin{equation*}
\lim_{n\rightarrow \infty} |t_n^j-t_n^l|+|x_n^j-x_n^l|=\infty.
\end{equation*}

Next, we prove  \eqref{PYTHA}. Recall that $W_n^l=W^{l-1}_n-U(-t_n^l)\psi_l(\cdot -x_n^l)$. Assume that relation \eqref{PYTHA} holds at rank $l-1$. Let $0\leq \lambda \leq 1$, using the weak convergence \eqref{Wphi1l} and expanding
$$
 \|W_n^l\|^2_{\dot{H}^{\lambda}}= \|U(t_n^l)W^l_n(\cdot +x_n^l)\|^2_{\dot{H}^{\lambda}}=\|U(t_n^l)W^{l-1}_n(\cdot +x_n^l)-\psi_l\|^2_{\dot{H}^{\lambda}}
$$
as an inner product it is easy to conclude \eqref{PYTHA} at rank $k=l$.

Finally, we consider the smallness property \eqref{SPWNL}. Since $C_l\leq C_1$, for all $l\in \mathbb{N}$, we have
$$
\sum_{l=1}^{\infty} (A_l)^{\frac{2(k+2)}{k}+\frac{4}{k+4}}\lesssim \sum_{l=1}^{\infty} (C_l)^{\frac{2}{k}+\frac{2}{k+4}}\|\psi^l\|^2_{H^1}\lesssim(C_1)^{\frac{2}{k}+\frac{2}{k+4}} \limsup_{n\rightarrow \infty} \|\phi_n\|^2_{H^{1}_x}<\infty.
$$
Hence $A_l\rightarrow 0$, as $l\rightarrow \infty$ (giving \eqref{UWNL}), and the proof is complete.
\end{proof}

\begin{lemma}[Energy Pythagorean expansion] Under the hypotheses of Theorem \ref{profdec}, for any $l\geq1$,
\begin{equation}\label{EPEx}
E[\phi_n]-\sum_{j=1}^{l}E[U(t_n^j)\psi^j(\cdot-x^j_n)] - E[W_n^l], \rightarrow 0,  \peqq \textrm{as} \peqq n\rightarrow \infty.
\end{equation}
\end{lemma}
\begin{proof}
By \eqref{PYTHA} it suffices to prove that for any $l\geq1$,
\begin{equation}\label{PYTHA2}
\int\phi_n^{k+2}(x)dx-\sum_{j=1}^{l}\int(U(t_n^j)\psi^j)^{k+2}(x) dx - \int(W_n^l)^{k+2}(x)dx \rightarrow 0,  \peqq \textrm{as} \peqq n\rightarrow \infty.
\end{equation}

First, we claim that for a fixed $l\geq1$,
\begin{equation}\label{PYTHA3}
\int\left(\sum_{j=1}^{l}U(t_n^j)\psi^j(x-x_n^j)\right)^{k+2}\!\!\!\!\!\!\!\!\!\!dx-\sum_{j=1}^{l}\int(U(t_n^j)\psi^j)^{k+2}(x) dx \rightarrow 0,  \peqq \textrm{as} \peqq n\rightarrow \infty.
\end{equation}
By reindexing and passing to a subsequence if necessary, without loss of generality, there exists $l_0\leq l$ such that
\begin{itemize}
\item [(i)] For $1\leq j \leq l_0$, the sequence $\{t_n^j\}_{n\in \mathbb{N}}$ is bounded, which implies 
$$t_n^j\rightarrow \bar{t}^j \in \R, \peqq \textrm{as} \peqq n\rightarrow \infty.$$

\item [(ii)] For $l_0+1\leq j\leq l$, we have that 
$$|t_n^j|\rightarrow \infty,  \peqq \textrm{as} \peqq n\rightarrow \infty.$$
\end{itemize}
Note that in case $(\textrm{ii})$, we have for all $l_0+1\leq j\leq l$
\begin{equation}\label{LIM2}
\lim_{n\rightarrow \infty}\int(U(t_n^j)\psi^j)^{k+2}(x) dx=0.
\end{equation}
Indeed, for $\psi \in L^{\frac{2(2k+4)}{3k+4}}_x\cap \dot{H}^{\frac{1}{2}-\frac{1}{k+2}}$, we have by Sobolev's embedding and Lemma \ref{lemma0}
\begin{equation*}
\begin{split}
\|U(t)\psi^j\|_{L^{k+2}}&=\|U(t)(\psi^j-\psi)\|_{L^{k+2}}+\|U(t)\psi\|_{L^{k+2}}\\
&\lesssim \|\psi^j-\psi\|_{\dot{H}^{\frac{1}{2}-\frac{1}{k+2}}} + \|D_x^{\frac{1}{r} -\frac{1}{k+2}}U(t)\psi\|_{L^{r}}\\
&\lesssim \|\psi^j-\psi\|_{\dot{H}^{\frac{1}{2}-\frac{1}{k+2}}} + |t|^{-\frac{k}{4(k+2)}}\|\psi\|_{L^{\frac{2(2k+4)}{3k+4}}},
\end{split}
\end{equation*}
where $r=\frac{4(k+2)}{k+4}$. Therefore, a density argument shows \eqref{LIM2}. 

On the other hand, in case $(\textrm{i})$, \eqref{XT} implies $\lim_{n\rightarrow \infty} |x_n^i-x_n^j|=\infty$, for $1\leq i < j\leq l_0$. Therefore all the cross terms in the expansion of 
$$\int_{\R}\left(\sum_{j=1}^{l}U(t_n^j)\psi^j(x-x_n^j)\right)^{k+2}\!\!\!\!\!\!\!\!\!\!dx$$ 
goes to zero as $n\to\infty$, which implies \eqref{PYTHA3}.

Next, by \eqref{PYTHA}, we conclude that $\{W_n^l\}_{n\in \mathbb{N}}$ is bounded in $H^1$, since the sequence $\{\phi_n\}_{n\in \mathbb{N}}$ has this property by hypotheses. Therefore, by Sobolev's embedding both sequences are uniformly bounded in $L^{k+2}(\R)$. Moreover, by using \eqref{UWNL} and the fact that  $\|W^l_n\|_{L^{k+2}_x}\leq \|U(t)W^l_n\|_{L_t^{\infty}L^{k+2}_x}$, we also obtain
\begin{equation}\label{I1a}
\limsup_{n\rightarrow \infty} \|W^l_n\|_{L^{k+2}_x} \rightarrow 0,  \peqq \textrm{as} \peqq l\rightarrow \infty.
\end{equation}

Now observe that 
\begin{eqnarray*}
\left|\int\phi_n^{k+2}(x)dx-\sum_{j=1}^{l}\int(U(t_n^j)\psi^j)^{k+2}(x) dx - \int(W_n^l)^{k+2}(x)dx\right|\leq I_1+I_2,
\end{eqnarray*}
where, for some $l_1\geq1$,
$$
I_1=\left|\int(\phi_n-W_n^{l_1})^{k+2}dx- \int\phi_n^{k+2}dx\right| +\left|\int(W^l_n-W_n^{l_1})^{k+2}dx- \int(W_n^l)^{k+2}dx\right|,
$$
$$
I_2=\left| \int(\phi_n-W_n^{l_1})^{k+2}dx-\sum_{j=1}^{l}\int(U(t_n^j)\psi^j)^{k+2} dx - \int(W^l_n-W_n^{l_1})^{k+2}dx\right|.
$$

We first estimate $I_1$. By recalling the inequality
$$
\big||a+b|^{k+2}-|a|^{k+2}\big|\lesssim |a||b|^{k+1}+|a|^{k+1}|b|+|b|^{k+2},
$$
a simple calculation and H\"older's inequality imply
\begin{eqnarray}\label{I1b}
I_1&\lesssim & \left(\int_{\R}|\phi_n|^{k+1}|W_n^{l_1}| + |W_n^{l_1}|^{k+1}|\phi_n|  +|W_n^{l}|^{k+1}|W_n^{l_1}| + |W_n^{l_1}|^{k+1}|W_n^{l}| dx    \right)\nonumber \\
&&+ 2\|W_n^{l_1}\|^{k+2}_{L^{k+2}}\nonumber \\
&\lesssim & \big(\sup_{n}\|\phi_n\|^{k+1}_{L^{k+2}}+\sup_n\|W_n^{l}\|^{k+1}_{L^{k+2}}\big)\|W_n^{l_1}\|_{L^{k+2}}\nonumber \\
&&+\big(\sup_{n}\|\phi_n\|_{L^{k+2}}+\sup_n\|W_n^{l}\|_{L^{k+2}}\big)\|W_n^{l_1}\|^{k+1}_{L^{k+2}} + 2\|W_n^{l_1}\|^{k+2}_{L^{k+2}}.
\end{eqnarray}

Therefore, combining  \eqref{I1a} and \eqref{I1b}, given any $\varepsilon>0$ there exist natural numbers $l_1>l$ and $n_1$ such that
\begin{eqnarray*}%\label{I1c}
I_1\leq \varepsilon, \peq \textrm{for all} \peq n\geq n_1.
\end{eqnarray*}

Next we consider  $I_2$. By definition, we have
$$
\phi_n-W_n^{l_1}=\sum_{j=1}^{l_1}U(t_n^j)\psi^j(\cdot-x_n^j)
$$
and
$$
W_n^{l_1}-W_n^{l}=\sum_{j=l+1}^{l_1}U(t_n^j)\psi^j(\cdot-x_n^j).
$$
However, by the proof of \eqref{PYTHA3}, there exists $n_2\in\mathbb{N}$ large enough such that
\begin{eqnarray*}%\label{I2}
I_2& \leq & \left| \int(\phi_n-W_n^{l_1})^{k+2}dx-\sum_{j=1}^{l_1}\int(U(t_n^j)\psi^j)^{k+2} dx\right| \nonumber \\
&&+\left| \sum_{j=l+1}^{l_1}\int(U(t_n^j)\psi^j)^{k+2} dx - \int(W^l_n-W_n^{l_1})^{k+2}dx\right|\nonumber \\
& \leq & \varepsilon, \peq \textrm{for all} \peq n\geq n_2.
\end{eqnarray*}

Finally, taking $n\geq \max\{n_1,n_2\}$ we see that $I_1+I_2\leq2\varepsilon$. This concludes the proof of \eqref{PYTHA2}.
\end{proof}

%%%%%%%%%%%%%%%%%%%%%%%%%%%%%%%%%%%%%%%%%%%%%%%%%%%%%%%%%%%%%%%%%%%%%%%%%%%%%%%%%%%%%%%%%%%%%%%%%%%%%%%%%%%%%%%%%%%%%%%%%%%%
%%%%%%%%%%%%%%%%%%%%%%%%%%%%%%%%%%%%%%%%%%%%%%%%%%%%%%%%%%%%%%%%%%%%%%%%%%%%%%%%%%%%%%%%%%%%%%%%%%%%%%%%%%%%%%%%%%%%%%%%%%%%
%%%%%%%%%%%%%%%%%%%%%%%%%%%%%%%%%%%%%%%%%%%%%%%%%%%%%%%%%%%%%%%%%%%%%%%%%%%%%%%%%%%%%%%%%%%%%%%%%%%%%%%%%%%%%%%%%%%%%%%%%%%%
%%%%%%%%%%%%%%%%%%%%%%%%%%%%%%%%%%%%%%%%%%%%%%%%%%%%%%%%%%%%%%%%%%%%%%%%%%%%%%%%%%%%%%%%%%%%%%%%%%%%%%%%%%%%%%%%%%%%%%%%%%%%
%%%%%%%%%%%%%%%%%%%%%%%%%%%%%%%%%%%%%%%%%%%%%%%%%%%%%%%%%%%%%%%%%%%%%%%%%%%%%%%%%%%%%%%%%%%%%%%%%%%%%%%%%%%%%%%%%%%%%%%%%%%%
%%%%%%%%%%%%%%%%%%%%%%%%%%%%%%%%%%%%%%%%%%%%%%%%%%%%%%%%%%%%%%%%%%%%%%%%%%%%%%%%%%%%%%%%%%%%%%%%%%%%%%%%%%%%%%%%%%%%%%%%%%%%
%%%%%%%%%%%%%%%%%%%%%%%%%%%%%%%%%%%%%%%%%%%%%%%%%%%%%%%%%%%%%%%%%%%%%%%%%%%%%%%%%%%%%%%%%%%%%%%%%%%%%%%%%%%%%%%%%%%%%%%%%%%%
%%%%%%%%%%%%%%%%%%%%%%%%%%%%%%%%%%%%%%%%%%%%%%%%%%%%%%%%%%%%%%%%%%%%%%%%%%%%%%%%%%%%%%%%%%%%%%%%%%%%%%%%%%%%%%%%%%%%%%%%%%%%
%%%%%%%%%%%%%%%%%%%%%%%%%%%%%%%%%%%%%%%%%%%%%%%%%%%%%%%%%%%%%%%%%%%%%%%%%%%%%%%%%%%%%%%%%%%%%%%%%%%%%%%%%%%%%%%%%%%%%%%%%%%%

\section{Existence and precompactness of a critical solution}\label{sec5}

In this section we will construct a critical solution in a sense given below. We follow the exposition in Kenig and Merle \cite{kenig-merle2}. We start with the following definitions.
\begin{definition}
For $A>0$, define
\begin{equation*}
B(A)=\left\{
 u_0\in H^1: M[u_0]+E[u_0]\leq A
\right\}.
\end{equation*}
\end{definition}

\begin{definition}
We shall say that $SC(A)$ holds if for each $u_0\in B(A)$, the corresponding solution $u(t)$ of \eqref{gkdv} is global, that is, $T_+(u_0)=+\infty$ and 
$$
\|u\|_{L^{5k/4}_xL^{5k/2}_t}<\infty.
$$
\end{definition}

Note that if $k$ is even then, by Theorem \ref{global5}, $T_+(u_0)=+\infty$ and $\sup_{t\in [0,+\infty)}\|u(t)\|_{H^1}<(M[u_0]+2E[u_0])^{1/2}<+\infty$ for all $u_0\in H^1$. Moreover, if  $\|u\|_{L^{5k/4}_xL^{5k/2}_t}<\infty$, then Proposition \ref{PROPSCAT} implies that $u$ scatters. Also, the small data theory (Proposition \ref{SDGT}) asserts that there exists $A_0>0$ such that $SC(A)$ holds for all $A\leq A_0$. Indeed, this is a consequence of the following inequality
$$
\| U(t) u_{0} \|_{L^{5k/4}_{x} L^{5k/2}_{t}}\leq \|u_0\|_{\dot{H}^{s_k}}\leq \|u_0\|_{{H}^{1}}\leq 2^{1/2} (M[u_0]+E[u_0])^{1/2},
$$
where we have used Lemma \ref{lemma12} and inequality \eqref{Apriori}.

Our goal now is to show that $SC(A)$ holds for any $A>0$, which is equivalent to Theorem \ref{global4}. If it fails, then there exists a critical number $A_C>A_0$ such that if $A<A_C$, $SC(A)$ holds, but if $A>A_C$, $SC(A)$ fails; more precisely,
$$
A_C=\sup\left\{A: \; u_0\in B(A) \;\Rightarrow \;\|u\|_{L^{5k/4}_xL^{5k/2}_t}<\infty \right\}.
$$
Thus, we can find $A_n\searrow A_C$ and a sequence of initial data $u_{n,0}\in H^1$ such that $u_{n,0}\in B(A_n)$. In this case, denoting by $u_n$ the corresponding solution of \eqref{gkdv}, we have that
$$
M[u_{n,0}]+E[u_{n,0}]\leq A_n
$$
and
\begin{equation}\label{uninfty}
\|u_n\|_{L^{5k/4}_xL^{5k/2}_{t}}=\infty,
\end{equation}
for all $n\in \mathbb{N}$. 

The next lemma is the main tool to construct a critical solution that does not scatters. It says that under certain conditions the profile decomposition given by Theorem \ref{profdec} contains at most one zero element.

\begin{lemma}\label{Lemma}
Suppose that $A_C<\infty$. Let $\{a_n\}_{n\in\N}$ such that $a_n\rightarrow A_C$, as $n\rightarrow +\infty$, and $\{\phi_n\}\subset H^1$ satisfying, for every $n\in \N$, $\phi_n\in B(a_n)$,
\begin{equation}\label{phininfty}
\|\textrm{KdV}(t)\phi_n\|_{L^{5k/4}_xL^{5k/2}_{t}}=\infty,
\end{equation}
where $\{\textrm{KdV}(t)\}_{t\in \R}$ denotes the flow of the generalized KdV equation \eqref{gkdv}. Then, up to a subsequence, there exist a function $\psi\in  H^1$ and sequences  $\{W_n\}_{n \in \mathbb{N}}\subset H^1$, $\{t_n\}_{ n \in \mathbb{N}}\subset \R$ and $\{x_n\}_{ n \in \mathbb{N}}\subset \R$, such that for every $n\geq1$
\begin{equation*}
\phi_n=U(t_n)\psi(\cdot-x_n) + W_n,
\end{equation*}
and
\begin{equation}\label{wnzero}
\lim_{n\rightarrow \infty}\|U(t)W_n\|_{L^{5k/4}_xL^{5k/2}_{t}}=0.
\end{equation}
\end{lemma}

\begin{proof}
Since $a_n\rightarrow A_C$ we can assume  $a_n\leq 2A_C$, for all $n\in \mathbb{N}$. Therefore, since $\phi_n\in B(a_n)$, we deduce that
\begin{equation}\label{partialphin}
\|\phi_n\|^2_{H^1}\leq 2(M[\phi_n]+E[\phi_n])\leq 4A_C<+\infty,
\end{equation}
which implies that $\{\phi_{n}\}_{n\in \mathbb{N}}$ is a uniformly bounded sequence in $H^1$ and we can apply the profile decomposition result (Theorem \ref{profdec}) to obtain, for any $l\geq1$,
\begin{equation}\label{decomp1}
\phi_{n}=\sum_{j=1}^{l}U(t_n^j)\psi^j(\cdot-x^j_n) + W_n^l.
\end{equation}

By the Pythagorean expansion \eqref{PYTHA}, with $\lambda=0$, we have for all $l\geq1$,
\begin{equation*}
\sum_{j=1}^{l}\|\psi^j\|^2_{L^{2}} + \limsup_{n\rightarrow \infty}\|W_n^l\|^2_{L^{2}} \leq \limsup_{n\rightarrow \infty}\|\phi_{n}\|^2_{L^{2}}.
\end{equation*}

Moreover, by the Energy Pythagorean expansion \eqref{EPEx}, we also have that
\begin{equation*}%\label{Energy-Dec}
\sum_{j=1}^{l}\limsup_{n\rightarrow \infty}E[U(t_n^j)\psi^j(\cdot-x^j_n)] + \limsup_{n\rightarrow \infty}E[W_n^l] \leq \limsup_{n\rightarrow \infty}E[\phi_{n}].
\end{equation*}

Collecting the last two inequalities we deduce\footnote{Note that in the defocusing case with $k$ even we have that $E[f]\geq 0$, for all $f\in H^1({\R})$.} 
\begin{equation}\label{L2Norm}
\sum_{j=1}^{l}(M[\psi^j]+\limsup_{n\rightarrow \infty}E[U(t_n^j)\psi^j(\cdot-x^j_n)])+\limsup_{n\rightarrow \infty}(M[W_n^l]+E[W_n^l])\leq A_C.
\end{equation}

Next we show that we cannot have more than one $\psi^j$ not zero. Indeed, suppose by contradiction that more than one $\psi^j$ is nonzero. Passing to a subsequence if necessary, we may suppose $t^j_n\rightarrow \bar{t}^j \in [-\infty,+\infty]$, as $n\to\infty$. Moreover,  from \eqref{L2Norm}, we deduce
\begin{equation*}%\label{psijAC}
M[\psi^j]+\limsup_{n\rightarrow \infty}E[U(t_n^j)\psi^j(\cdot-x^j_n)]< A_C,
\end{equation*}
for every $j\geq1$. 

Now, let $\bar{u}^j$ be the nonlinear profile associated with $(\psi^j, \{ t^j_{n} \}_{n \in  \mathbb{N}})$ (see Definition \ref{NLP} and Remark \ref{4.2}).
We claim that
\begin{equation}\label{EXPLA}
\|\bar{u}^j\|_{L^{5k/4}_xL^{5k/2}_t}<\infty. 
\end{equation}
In fact, by definition, we have
\begin{equation}\label{NLP3}
\lim_{n \to \infty} \| \bar{u}^j(t^j_{n}) - U(t^j_{n}) \psi^j \|_{{H}^{1}} = 0.
\end{equation}
Also, from \eqref{NLP3} it is clear that $M[\bar{u}^j]=M[\psi^j]$. Next, we control the quantity $E[\bar{u}^j]$. We have two cases to consider: $|\bar{t}^j|=+\infty$ or $\bar{t}^j\in \R$. In the first case, by the same argument used to prove \eqref{LIM2} we obtain $\underset{n\rightarrow \infty}{\lim}\|U(t_n^j)\psi^j(\cdot-x^j_n)\|_{L^{k+2}}=0$. Therefore, in view of \eqref{NLP3} and the Sobolev Embedding $H^1(\R) \hookrightarrow L^{k+2}(\R)$, for all $k\in \N$, we have  
$\underset{n\rightarrow \infty}{\lim}\|\bar{u}^j(t_n^j)\|_{L^{k+2}}=0$. Moreover 
\[
\begin{split}
E[\bar{u}^j]&=\lim_{n\rightarrow \infty}E[\bar{u}^j(t_n^j)]=\lim_{n\rightarrow \infty}\frac{1}{2}\|\partial_x\bar{u}^j(t_n^j)\|_{L^2}=\frac{1}{2}\|\partial_x \psi^j\|^2_{L^2}\\
&=\lim_{n\rightarrow \infty} E[U(t_n^j)\psi^j(\cdot-x^j_n)],
\end{split}
\]
where we used that $\|\partial_x U(t_n^j)\psi^j(\cdot-x^j_n)\|_{L^2}=\|\partial_x \psi^j\|_{L^2}$.
On the other hand, if $\bar{t}^j\in \R$, by the continuity of the linear flow and \eqref{NLP3} we also have
$$
E[\bar{u}^j]=E[U(\bar{t}^j) \psi^j]=\lim_{n\rightarrow \infty} E[U(t_n^j)\psi^j(\cdot-x^j_n)].
$$
Thus, in both cases,
\begin{equation}\label{EXPLA1}
M[\bar{u}^j]+E[\bar{u}^j]< A_C.
\end{equation}
Therefore, by the definition of $A_C$ we deduce \eqref{EXPLA}.

Next we claim that 
\begin{equation}\label{Ds510}
\|D^{s_k}_x\bar{u}^j\|_{L^{5}_xL^{10}_{t}}<\infty.
\end{equation}
Indeed, the proof follows the ideas in Proposition \ref{PROPSCAT}. Let $\delta>0$ be a small number to be chosen later. From \eqref{EXPLA} we can decompose the $[0,\infty)$ into a finite number of intervals, say, $I_n=[t_n,t_{n+1})$, $n=0,1,\dots,\ell-1$  such that
$\|\bar{u}^j\|_{L^{5k/4}_xL^{5k/2}_{I_n}}<\delta$. From \eqref{eq4} with $(p,q,\alpha)=(5,10,0)$, we have, for $n=0,1,\dots,\ell-1$,
\begin{equation*}%\label{eq9.1}
\begin{split}
\|D_x^{s_k}\bar{u}^j&\|_{L^5_xL^{10}_{I_n}}\lesssim
\|D_x^{s_k}\bar{u}^j(0)\|_{L^2}+\|D^{s_k}_x\int_0^tU(t-t')\partial_x((\bar{u}^j)^{k+1})(t')dt'\|_{L^5_xL^{10}_{I_n}}\\
&\lesssim\|D_x^{s_k}\bar{u}^j(0)\|_{L^2}+\sum_{m=0}^n\|D_x^{s_k}\int_{t_m}^{t_{m+1}}U(t-t')\partial_x((\bar{u}^j)^{k+1})(t')dt'\|_{L^5_xL^{10}_{I_n}}\\
&\lesssim \|D_x^{s_k}\bar{u}^j(0)\|_{L^2}+\sum_{m=0}^n\|D_x^{s_k}\int_{0}^{t}U(t-t')\partial_x((\bar{u}^j)^{k+1})(t')\chi_{I_m}(t')dt'\|_{L^5_xL^{10}_{t}},\\
\end{split}
\end{equation*}
where $\chi_{I_m}$ denotes the characteristic function of the interval $I_m$.
By using Lemma \ref{lemma2}, with $(p_1,q_1,\alpha_1)=(5,10,0)$ and
$(p_2,q_2,\alpha_2)=(\infty,2,1)$, and the Leibnitz rule, we then deduce
\begin{equation*}%\label{eq9.3}
\begin{split}
\|D_x^{s_k}\bar{u}^j\|_{L^5_xL^{10}_{I_n}}&\lesssim \|D_x^{s_k}\bar{u}^j(0)\|_{L^2}+\sum_{m=0}^n\|D_x^{s_k}((\bar{u}^j)^{k+1})\|_{L^1_xL^{2}_{I_m}}\\
&\lesssim \|D_x^{s_k}\bar{u}^j(0)\|_{L^2}+\sum_{m=0}^n\|D_x^{s_k}\bar{u}^j\|_{L^5_xL^{10}_{I_m}}\|\bar{u}^j\|_{L^{5k/4}_xL^{5k/2}_{I_m}}^k\\
&\leq
c\|D_x^{s_k}\bar{u}^j(0)\|_{L^2}+c\delta^k\sum_{m=0}^n\|D_x^{s_k}\bar{u}^j\|_{L^5_xL^{10}_{I_m}}.
\end{split}
\end{equation*}
Therefore, by choosing $c\delta^k<1/2$, we conclude
\begin{equation}\label{eq9.3a}
\begin{split}
\|D_x^{s_k}\bar{u}^j\|_{L^5_xL^{10}_{I_n}}\leq
2c\|D_x^{s_k}\bar{u}^j(0)\|_{L^2}+\sum_{m=0}^{n-1}\|D_x^{s_k}\bar{u}^j\|_{L^5_xL^{10}_{I_m}}.
\end{split}
\end{equation}
Inequality \eqref{eq9.3a} together with a induction argument implies that
$\|D_x^{s_k}\bar{u}^j\|_{L^5_xL^{10}_{I_n}}<\infty$, $n=0,1,\dots,\ell-1$. By summing
over the $\ell$ intervals we conclude $\|D_x^{s_k}\bar{u}^j\|_{L^5_xL^{10}_{[0,+\infty)}}<\infty$. Using the same argument in the interval $(-\infty,0]$ we deduce \eqref{Ds510}. In particular, \eqref{eq9.3a} implies that
\begin{equation}\label{eq9.3b}
\|D^{s_k}_x\bar{u}^j\|_{L^5_xL^{10}_t}\lesssim \|D_x^{s_k}\bar{u}^j(0)\|_{L^2}\lesssim \|\bar{u}^j(0)\|_{H^1}.
\end{equation}

Next, define
$$v_n(t)=\textrm{KdV}(t)\phi_{n},$$
$$v_n^j(t)=\bar{u}^j( t+t_n,\cdot-x_n)$$
and
\begin{equation}\label{unl}
u_n^l(t)=\sum_{j=1}^lv_n^j(t).
\end{equation}
We can easily check that $u_n^l$ satisfies the following equation
$$
\partial_t {u}_n^l+\partial_{xxx}{u}_n^l+\partial_x({u}_n^l)^{k+1}=\partial_xe_n^l,
$$
where $e_n^l=(u_n^l)^{k+1}-\sum_{j=1}^l(v_n^j)^{k+1}.$ Moreover, by setting 
$$
\widetilde{W}^l_n=W^l_n - \sum_{j=1}^l\Big(\bar{u}^j(t^j_n,\cdot-x^j_n)-U(t^j_n)\psi(\cdot-x^j_n)\Big),
$$
from Lemma \ref{lemma12},  the definition of the nonlinear profile and  \eqref{SPWNL}, we have
$$\limsup_{n\to\infty}\|U(t)\widetilde{W}^l_n\|_{L^{5k/4}_xL^{5k/2}_t} = \limsup_{n\to\infty}\|U(t){W}^l_n\|_{L^{5k/4}_xL^{5k/2}_t}\rightarrow 0, \peqq \textrm{as} \peqq l\rightarrow \infty.$$
It is clear that $v_n(0)-u_n^l(0)=\widetilde{W}^l_n$ and therefore, given $\varepsilon>0$, we have for $n$ and $l$ sufficiently large,
$$\|U(t)\big(v_n(0)-u_n^l(0)\big)\|_{L^{5k/4}_xL^{5k/2}_t}<\varepsilon.$$

The idea now is to obtain a relation between $v_n$ and $u_n^l$ using the perturbation theory (Proposition \ref{LTPT}). To this end we first claim that for a fixed $l$ there exists $n_0(l)\in \mathbb{N}$ such that, for $n\geq n_0(l)$,
\begin{equation}\label{claimen}
\|D^{s_k}_xe_n^l\|_{L^1_xL^2_t}+\|D_xe_n^l\|_{L^1_xL^2_t}+\|e_n^l\|_{L^1_xL^2_t}<\varepsilon_1,
\end{equation}
where $\varepsilon_1>0$ is given in Proposition \ref{LTPT}.

Indeed, we start with the norm $\|D^{s_k}_xe_n^l\|_{L^1_xL^2_t}$. The expansion of $e_n^l$ consists of $\sim$ $l^{k+1}$ cross terms of the form
$$v_n^{j_1}\cdots v_n^{j_{k+1}}, \peq \textrm{with} \peq j_1,\dots j_{k+1}\in \{1,\dots, l\},$$
where at least two indices $j_i$ are different. So, it suffices to show that each one of these terms goes to zero, as $n\to\infty$. Assume, without loss of generality, that $j_1\neq j_2$. By the Leibnitz rule (Lemma \ref{lemma6})
\begin{equation}\label{estivj}
\begin{split}
\|D_x^{s_k}(v_n^{j_1}v_n^{j_2}v_n^{j_3}\cdots v_n^{j_{k+1}})\|_{L^1_xL^2_t}&\leq \|D_x^{s_k}(v_n^{j_1}v_n^{j_2})\|_{L^{p_1}_xL^{q_1}_t}\|v_n^{j_3}\cdots v_n^{j_{k+1}}\|_{L^{p_2}_xL^{q_2}_t}\\
&\quad+\|v_n^{j_1}v_n^{j_2}\|_{L^{\widetilde{p}_1}_xL^{\widetilde{q}_1}_t}\|D_x^{s_k}(v_n^{j_3}\cdots v_n^{j_{k+1}})\|_{L^{\widetilde{p}_2}_xL^{\widetilde{q}_2}_t},
\end{split}
\end{equation}
where $p_1=5k/(k+4)$, $q_1=10k/(k+4)$, $\widetilde{p}_1=5k/8$, $\widetilde{q}_1=5k/4$.
By using H\"older's inequality and the Leibnitz rule, it is not difficult to see that
\begin{equation}\label{estivj1}
\|v_n^{j_3}\cdots v_n^{j_{k+1}}\|_{L^{p_2}_xL^{q_2}_t}\lesssim \prod_{i=3}^{k+1}\|v_n^{j_i}\|_{L^{5k/4}_xL^{5k/2}_t}
\end{equation}
and
\begin{equation}\label{estivj2}
\|D_x^{s_k}(v_n^{j_3}\cdots v_n^{j_{k+1}})\|_{L^{\widetilde{p}_2}_xL^{\widetilde{q}_2}_t}\lesssim \sum_{i=3}^{k+1}\|D^{s_k}_xv_n^{j_i}\|_{L^5_xL^{10}_t}\prod_{m\neq i}\|v_n^{j_m}\|_{L^{5k/4}_xL^{5k/2}_t}.
\end{equation}
In order to verify that $\|D^{s_k}_xe_n^l\|_{L^1_xL^2_t}$ is small, it suffices to prove that the r.h.s of \eqref{estivj} goes to zero, as $n\to\infty$. To do so, it suffices to show that the r.h.s of \eqref{estivj1} and \eqref{estivj2} are uniformly bounded (with respect to $n$) and 
\begin{equation}\label{LIMIT}
\|D_x^{s_k}(v_n^{j_1}v_n^{j_2})\|_{L^{5k/(k+4)}_xL^{10k/(k+4)}_t} +\|v_n^{j_1}v_n^{j_2}\|_{L^{5k/8}_xL^{5k/4}_t} \rightarrow 0, \peqq \textrm{as} \peqq n\rightarrow \infty.
\end{equation}

Since the norm in $L^{5k/4}_xL^{5k/2}_t$ is invariant by translations, from the definition of $v_n^{j_i}$, it is clear that $\|v_n^{j_i}\|_{L^{5k/4}_xL^{5k/2}_t}$ is uniformly bounded. On the other hand, from  \eqref{eq9.3b},
\[
\begin{split}
\|D^{s_k}_xv_n^{j_i}\|_{L^5_xL^{10}_t}\lesssim \|\bar{u}^j(t_n,\cdot-x_n)\|_{H^1}\lesssim \sup_{t \in \R} \|\bar{u}^j(t)\|_{H^1}\lesssim (M[\bar{u}^j(0)]+2E[\bar{u}^j(0)])^{1/2}.
\end{split}
\]
Thus, it remains to show \eqref{LIMIT}.
Since $v_n^{j_1},v_n^{j_2} \in L^{5k/4}_xL^{5k/2}_t$, without loss of generality, we can assume that $v_n^{j_1},v_n^{j_2} \in C_0^{\infty}(\R^2)$. Using property \eqref{XT}, we promptly infer that
$$\|v_n^{j_1}v_n^{j_2}\|_{L^{5k/8}_xL^{5k/4}_t} \rightarrow 0, \peqq \textrm{as} \peqq n\rightarrow \infty.$$
Next, we turn attention to the other norm that appears in \eqref{LIMIT}. Note that $v_n^{j_1},v_n^{j_2}, \partial_xv_n^{j_1}, \partial_x v_n^{j_2} \in L^{5}_xL^{10}_t$, therefore property \eqref{XT} yields
\begin{equation}\label{ERROR1}
\|v_n^{j_1}v_n^{j_2}\|_{L^{5k/(k+4)}_xL^{10k/(k+4)}_t} \rightarrow 0, \peqq \textrm{as} \peqq n\rightarrow \infty.
\end{equation}
and
\begin{equation}\label{ERROR2}
\begin{split}
\|\partial_x(v_n^{j_1}v_n^{j_2})\|_{L^{5k/(k+4)}_xL^{10k/(k+4)}_t}\leq & \|\partial_x(v_n^{j_1})v_n^{j_2}\|_{L^{5k/(k+4)}_xL^{10k/(k+4)}_t} \\
&+ \|v_n^{j_1}\partial_x(v_n^{j_2})\|_{L^{5k/(k+4)}_xL^{10k/(k+4)}_t}\rightarrow 0, \peqq \textrm{as} \peqq n\rightarrow \infty.
\end{split}
\end{equation}

Now, by complex interpolation,
$$\|D_x^{s_k}(v_n^{j_1}v_n^{j_2})\|_{L^{5k/(k+4)}_xL^{10k/(k+4)}_t}\leq \|v_n^{j_1}v_n^{j_2}\|_{L^{5k/(k+4)}_xL^{10k/(k+4)}_t}^{1-s_k}\|\partial_x(v_n^{j_1}v_n^{j_2})\|_{L^{5k/(k+4)}_xL^{10k/(k+4)}_t}^{s_k},$$
which implies \eqref{LIMIT}. The other norms $\|D_xe_n^l\|_{L^1_xL^2_t}$ and $\|e_n^l\|_{L^1_xL^2_t}$ can be estimated analogously using relations \eqref{ERROR1} and \eqref{ERROR2}. Gathering these estimates together yields  \eqref{claimen}.

Next, we show that there exist $L>0$ (independent of $l$) and $n_1(l)\in\mathbb{N}$ such that
\begin{equation}\label{CLAIM}
\|u_n^l\|_{L^{5k/4}_xL^{5k/2}_t}<L, \peq \textrm{for all} \peq n>n_1(l).
\end{equation}
By the Pythagorean expansion \eqref{PYTHA}, with $\lambda=1$, and \eqref{partialphin}, we have for all $l\in \mathbb{N}$,
$$
\sum_{j=1}^{l}\|\partial_x\psi^j\|^2_{L^{2}} + \limsup_{n\rightarrow \infty}\|\partial_xW_n^l\|^2_{L^{2}} \leq \limsup_{n\rightarrow \infty}\|\partial_x\phi_{n}\|^2_{L^{2}} \leq 2^{1+1/s_k}A_C^{1/s_k},
$$
which implies, using also \eqref{L2Norm}, the existence of a constant $C_1>0$ such that
$$\sum_{j=1}^{\infty}\|\psi^j\|^2_{H^1}\leq C_1.$$
Therefore, we can find $l_0\in \mathbb{N}$ large enough satisfying
$$\sum_{j=l_0}^{\infty}\|\psi^j\|^2_{H^1}\leq \delta,$$
where $\delta>0$ is a sufficiently small number to be chosen later.

Fix $l> l_0$. From the construction of the nonlinear profile $\bar{u}^j$, there exists $n_1(l)\geq 1$ such that, for all $n\geq n_1(l)$,
\begin{equation*}%\label{KDVTNJ}
\sum_{j=l_0}^{l}\|\bar{u}^j( t_n^j)\|^2_{H^1}\leq 2\delta.
\end{equation*}
Therefore, choosing $\delta>0$ sufficiently small such that we can apply Proposition \ref{SDGT} for all $n\geq n_1(l)$, we deduce that
\begin{equation}\label{KDVTNJ1}
\sum_{j=l_0}^{l}\|v_n^j\|^2_{L^{5k/4}_xL^{5k/2}_t}\leq 2 \sum_{j=l_0}^{l}\|\bar{u}^j(\cdot-x_n^j, t_n^j)\|^2_{H^1}\leq 4\delta.
\end{equation}

Now by the definition of $u^l_n$ (see \eqref{unl}) and the elementary inequality
$$
\Big||\sum_{j=1}^{l}z_j|^a-\sum_{j=1}^{l}|z_j|^a\Big|\leq C_l\sum_{i\neq m}|a_i||z_m|^{a-1}, \qquad a>1,
$$
we have
\[
\begin{split}
\|u_n^l\|_{L^{5k/4}_xL^{5k/2}_t}^{5k/4}&=\int \left( \int|u_n^l|^{5k/2}dt\right)^{1/2}dx\\
&\leq \int \left( \int\Big||\sum_{j=1}^lv_n^j|^{5k/2}-\sum_{j=1}^l|v_n^j|^{5k/2}\Big|+\sum_{j=1}^l|v_n^j|^{5k/2}dt\right)^{1/2}\!\!\!\!dx\\
&\lesssim \int \left( \int\sum_{j=1}^l|v_n^j|^{5k/2}dt \right)^{1/2}\!\!\!\!dx+\int \left( \int\sum_{i\neq m}|v_n^i||v_n^m|^{(5k-2)/2}dt \right)^{1/2}\!\!\!\!dx\\
&= \sum_{j=1}^{l}\|v_n^j\|_{L^{5k/4}_xL^{5k/2}_t}^{5k/4}+\textrm{cross terms},
\end{split}
\]
that is,
\begin{equation}\label{crossl}
\|u_n^l\|_{L^{5k/4}_xL^{5k/2}_t}^{5k/4}=\sum_{j=1}^{l_0-1}\|v_n^j\|_{L^{5k/4}_xL^{5k/2}_t}^{5k/4}+\sum_{j=l_0}^{l}\|v_n^j\|_{L^{5k/4}_xL^{5k/2}_t}^{5k/4}+\textrm{cross terms}.
\end{equation}

The ``cross terms" can be made small for $n \in \mathbb{N}$ large (depending on $l$). Indeed, if $1/p+1/q=1$,
\[
\begin{split}
\textrm{cross terms}&\lesssim \sum_{i\neq m}\int\left( \int|v_n^i||v_n^m|^{\frac{5k-2}{2}}dt \right)^{1/2}dx\\
&=\sum_{i\neq m}\int\left( \int|v_n^i||v_n^m||v_n^m|^{(5k-4)/2}dt \right)^{1/2}dx\\
&\leq\sum_{i\neq m}\int\left( \|v_n^iv_n^m\|_{L^q_t}\|v_n^m\|^{(5k-4)/2}_{L_t^{p(5k-4)/2}} \right)^{1/2}dx\\
&=\sum_{i\neq m}\int\left( \|v_n^iv_n^m\|_{L^q_t}^{1/2}\|v_n^m\|^{(5k-4)/4}_{L_t^{p(5k-4)/2}} \right)dx\\
&\leq \sum_{i\neq m}\|v_n^iv_n^m\|_{L_x^{q/2}L^q_t}^{1/2} \|v_n^m\|^{(5k-4)/4}_{L_x^{p(5k-4)/4}L_t^{p(5k-4)/2}}.
\end{split}
\]
By choosing $p=5k/(5k-4)$ we have $q=5k/4$. Thus,
\begin{equation}\label{crossl1}
\begin{split}
\textrm{cross terms}
\lesssim\sum_{i\neq m}\|v_n^iv_n^m\|_{L_x^{5k/8}L^{5k/4}_t}^{1/2} \|v_n^m\|^{(5k-4)/4}_{L_x^{5k/4}L_t^{5k/2}}.
\end{split}
\end{equation}
 the first term on the right-hand side of \eqref{crossl1} goes to zero as in \eqref{LIMIT}.  The remainder term is bounded by \eqref{EXPLA}.
Consequently, taking into account that $5k/4>2$, by using \eqref{KDVTNJ1} and \eqref{crossl1}, we get \eqref{CLAIM} from \eqref{crossl}.

A similar argument also establishes the existence of $M>0$ independent of $l$ and $n_2(l)\in\mathbb{N}$ such that
\begin{equation*}%\label{CLAIM2}
\|u_n^l\|_{L^{\infty}_tH^1}<M, \peq \textrm{for all} \peq n>n_2(l).
\end{equation*}

Finally, we apply Proposition \ref{LTPT} to $v_n$ and $u_n^l$, which in view of \eqref{CLAIM} leads to a contradiction with \eqref{phininfty} for sufficiently large $n\in\mathbb{N}$.  This indeed shows that the sum in \eqref{decomp1} must have a unique element, which without loss of generality we may assume to be the first one. Hence, by taking $t_n=t_n^1$, $x_n=x_n^1$, $\psi=\psi^1$, and $W_n=W_n^1$  we obtain the desired. Note that \eqref{wnzero} holds in view of  Theorem \ref{profdec} (see relation \eqref{SPWNL}).
\end{proof}

Now, we are in position to prove the following fundamental result.

\begin{theorem}[Existence of a critical solution]\label{ECE} 
There exists $u_{C,0}\in H^1$ such that  the corresponding global solution $u_C$ satisfies  $\|u_C\|_{L^{5k/4}_xL^{5k/2}_{t}}=\infty$. Moreover
\begin{equation}\label{uC}
M[u_C]+E[u_C]=A_C.
\end{equation}
\end{theorem} 
\begin{proof}
Taking $a_n=A_n$ and $\phi_n=u_{n,0}$ (where $A_n$ and $u_{n,0}$ are the sequences described in the introduction of this section) in Lemma \ref{Lemma}, there exist a function $\psi\in  H^1$ and sequences  $\{W_n\}_{n \in \mathbb{N}}\subset H^1$, $\{t_n\}_{ n \in \mathbb{N}}\subset \R$ and $\{x_n\}_{ n \in \mathbb{N}}\subset \R$, such that for any $n\geq1$,
\begin{equation*}
u_{n,0}=U(t_n)\psi(\cdot-x_n) + W_n.
\end{equation*} 
Moreover, arguing as in \eqref{L2Norm}, it is easy to see that
$$
M[\psi]+\limsup_{n\rightarrow \infty} E[U(t_n)\psi(\cdot-x_n)]\leq A_C.
$$
Consider the sequence $\{t_n\}_{n\in \mathbb{N}}$. Let $\bar{u}$ is the nonlinear profile associated with $(\psi, \{ t_{n} \}_{n \in  \mathbb{N}})$ then
\begin{equation}\label{NLP2}
\lim_{n \to \infty} \| \bar{u}(t_{n}) - U(t_{n}) \psi \|_{{H}^{1}} = 0.
\end{equation}
As in the proof of \eqref{EXPLA1}, we can also deduce that
\begin{equation}\label{EXPLA2}
M[\bar{u}]+E[\bar{u}]\leq  A_C.
\end{equation}

Now, define $\widetilde{W}_n=W_n - (\bar{u}(t_n,\cdot-x_n)-U(t_n)\psi(\cdot-x_n))$. By using  Lemma \ref{lemma12}, we have 
\begin{equation*}
\begin{split}
\|U(t)\widetilde{W}_n\|_{L^{5k/4}_xL^{5k/2}_t} \leq \|U(t){W}_n\|_{L^{5k/4}_xL^{5k/2}_t}+c\|\bar{u}(t_{n}) - U(t_{n}) \psi\|_{\dot{H}^{s_k}},
\end{split}
\end{equation*}
which, in view of \eqref{wnzero} and \eqref{NLP2}, implies
\begin{equation}\label{wtilde0}
\lim_{n\to\infty}\|U(t)\widetilde{W}_n\|_{L^{5k/4}_xL^{5k/2}_t}\rightarrow 0.
\end{equation}

By the definition of $\widetilde{W}^l_n$ we can write
\begin{equation}\label{star2}
u_{n,0}=\bar{u}(t_n,\cdot-x_n)+\widetilde{W}_n.
\end{equation}

Let $u_C(t,x)=\bar{u}(t,x)$, we claim that 
\begin{equation}\label{UCL}
\|u_C\|_{L^{5k/4}_xL^{5k/2}_t}=\infty.
\end{equation}
Indeed, let $\widetilde{u}_n(t)=\bar{u}( t+t_n,\cdot-x_n)$. Then, it is clear that $\widetilde{u}_n$ solves the gKdV equation \eqref{gkdv}. Moreover, if \eqref{UCL} does not hold, we have
$$
\|\widetilde{u}_n\|_{L^{5k/4}_xL^{5k/2}_t}=\|u_C\|_{L^{5k/4}_xL^{5k/2}_t}<\infty.
$$
In addition, in view of \eqref{EXPLA2},
\begin{equation}\label{uCtn}
\sup_{t \in \R}\|\widetilde{u}_n(t)\|_{H^1}=\sup_{t \in \R}\|u_C(t)\|_{H^1}\leq (2A_C)^{1/2} <+\infty,
\end{equation}
for every $n\in \N$.

Now recall that $u_n(t)=\textrm{KdV}(t)u_{n,0}$. Our aim is to compare $\widetilde{u}_n$ and ${u}_n$ via Proposition \ref{LTPT} (perturbation theory). We have, from \eqref{wtilde0} and \eqref{star2},
\begin{eqnarray*}
\|U(t)(u_n(0)-\widetilde{u}_n(0))\|_{L^{5k/4}_xL^{5k/2}_t}&=&\|U(t)(u_{n,0}-\bar{u}(t_n,\cdot-x_n))\|_{L^{5k/4}_xL^{5k/2}_t} \\
&=&\|U(t)\widetilde{W}_n\|_{L^{5k/4}_xL^{5k/2}_t}\rightarrow 0, \peqq \textrm{as} \peqq n\rightarrow \infty.
\end{eqnarray*}
Therefore, applying Proposition \ref{LTPT} with $e=0$, we obtain $\|u_{n}\|_{L^{5k/4}_xL^{5k/2}_t}<\infty$ for $n\in \mathbb{N}$ large enough, which is a contradiction with \eqref{uninfty}. Finally, from \eqref{EXPLA2}, \eqref{UCL} and the definition of $A_C$, we also conclude \eqref{uC}.
\end{proof}

Let $u_C$ be the critical solution given by Theorem \ref{ECE}. Since $\|u_C\|_{L^{5k/4}_xL^{5k/2}_{t}}=\infty$ at least one of the following hold $\|u_C\|_{L^{5k/4}_xL^{5k/2}_{[0,+\infty)}}=\infty$ or $\|u_C\|_{L^{5k/4}_xL^{5k/2}_{(-\infty,0]}}=\infty$. The next result shows that this solution has a compactness property up to translation in space.

\begin{proposition}[Precompactness of the critical flow]\label{Precomp}
Let $u_C$ be the critical solution constructed in Theorem \ref{ECE} and assume $\|u_C\|_{L^{5k/4}_xL^{5k/2}_{[0,+\infty)}}=\infty$. Then there exists a continuous path $x\in C([0,+\infty):\R)$ such that the set
$$
\mathcal{B}:=\{u_C(t,\cdot-x(t)): t\geq 0\}\subset H^1
$$
is precompact in $H^1$. A corresponding conclusion is reached if $\|u_C\|_{L^{5k/4}_xL^{5k/2}_{(-\infty,0]}}=\infty$.
\end{proposition}
\begin{proof}
The proof is similar to that of Proposition 3.2 in \cite{dhr}. So, we only give the main steps. Assume first $\|u_C\|_{L^{5k/4}_xL^{5k/2}_{[0,+\infty)}}=\infty$. In $H^1$ we let $G\simeq\R$ act as a translation group, that is,
$$
x_0\cdot f=f(\cdot-x_0).
$$
Thus, in the quotient space $H^1/G$ we can introduce the metric
$$
d([f],[g]):=\inf_{x_0\in\R}\|f(\cdot-x_0)-g\|_{H^1},
$$
in such a way that the proof of the proposition is equivalent to establish that the set
$$
\mathcal{C}:=\pi(\{u_C(t), \;t\geq0\})
$$
is precompact in $H^1/G$, where $\pi:H^1\to H^1/G$ is the standard projection.

Now assume, by contradiction, that $\mathcal{C}$ is not precompact in $H^1/G$. Then, we can find a sequence $\{u_C(t_n)\}_{n\in\N}$ and $\varepsilon>0$ such that
\begin{equation}\label{contrGcom}
\inf_{x_0\in\R^3}\|u_C(t_n,\cdot-x_0)-u_C(t_m,\cdot)\|_{H^1}>\varepsilon,
\end{equation}
for all $m,n\in\N$ with $m\neq n$. To obtain a contradiction with \eqref{contrGcom}, it suffices to prove that there exist   a subsequence of $\{u_C(t_n)\}_{n\in\N}$, which we still denote by $\{u_C(t_n)\}_{n\in\N}$, a sequence $\{x_n\}_{ n \in \mathbb{N}}\subset \R$ and a function $v\in H^1$ such that $u_C(\cdot+x_n, t_n)\rightarrow v$ in $H^1$, as $n\rightarrow \infty$. If $\{t_n\}_{n\in\N}$ is bounded, up to a subsequence, we have $t_n\rightarrow \bar{t}\in [0,+\infty)$. So, by the continuity of the solution in $H^1$, the result follows taking $v=u_C(\bar{t})$ and $x_n=0$, for all $n\in \N$. Therefore, we only need to consider the case when $t_n\rightarrow +\infty$.

We already know that $M[u_C]+E[u_C]=A_C$, which implies that $\{u_C(t_n)\}$ is bounded in $H^1$. Applying Lemma \ref{Lemma} with $a_n=A_C$ and $\phi_n=u_{C}(t_n)$, there exist a function $\psi\in  H^1$ and sequences $\{W_n\}_{n \in \mathbb{N}}\subset H^1$, $\{t^1_n\}_{ n \in \mathbb{N}}\subset \R$ and $\{x_n\}_{ n \in \mathbb{N}}\subset \R$, such that for every $n\geq1$,
\begin{equation}\label{ucdecom}
u_{C}(t_n)=U(t^1_n)\psi(\cdot-x_n) + W_n.
\end{equation} 

By the Pythagorean expansion for the mass and energy (see \eqref{PYTHA} with $\lambda=0$ and \eqref{EPEx}), up to a subsequence, we have
$$
M[\psi]+ \lim_{n\rightarrow \infty}M[W_n] = \lim_{n\rightarrow \infty}M[u_{C}(t_n)] =M[u_{C}],
$$
and 
$$
\lim_{n\rightarrow \infty}E[U(t^1_n)\psi(\cdot-x_n)]+ \lim_{n\rightarrow \infty}E[W_n] = \lim_{n\rightarrow \infty}E[u_{C}(t_n)] =E[u_{C}].
$$

We claim that $\lim_{n\rightarrow \infty}M[W_n]= 0$ and $\lim_{n\rightarrow \infty}E[W_n]= 0$. In fact,  on the contrary, we have $M[\psi]+\lim_{n\rightarrow \infty}E[U(t^1_n)\psi(\cdot-x_n)]<A_C$. Thus, if $\bar{u}$ is the nonlinear profile associated with $(\psi,\{t_n^1\})$, we can argue as in \eqref{EXPLA} to obtain $\|\bar{u}\|_{L^{5k/4}_xL^{5k/2}_t}<\infty$, which in turn, following the steps in the proof of Theorem \ref{ECE}, leads to a contradiction.

So, since $\|\partial_x W_n\|\leq 2E[W_n]$, we have
\begin{equation}\label{Wn}
\lim_{n\to\infty}\|W_n\|_{H^1}=0.
\end{equation} 

Our goal now is to show that $\{t_n^1\}_{n\in\N}$ has a convergent subsequence. Indeed, if $t_n^1\rightarrow -\infty$, then 
\begin{align*}
\|U(t)u_C(t_n)\| _{L^{5k/4}_{x} L^{5k/2}_{(-\infty,0]}}\leq &\,\,\,\|U(t+t^1_n)\psi(\cdot-x_n)\| _{L^{5k/4}_{x} L^{5k/2}_{(-\infty,0]}}+\|U(t)W_n\| _{L^{5k/4}_{x} L^{5k/2}_{(-\infty,0]}}\\
\leq &\,\,\, \|U(t)\psi\| _{L^{5k/4}_{x} L^{5k/2}_{(-\infty,t^1_n]}}+\|W_n\| _{H^1}\rightarrow  \,\,\,0, \peq \mbox{as}\peq n\to\infty,
\end{align*}
 by \eqref{Wn} and the fact that $t_n^1\rightarrow -\infty$.
So, since $\underset{n\in \N}{\sup} \|u_C(t_n)\|_{H^1}\leq (2A_C)^{1/2}$ (see \eqref{uCtn}), for $n$ large we have
$$
\|U(t)u_C(t_n)\| _{L^{5k/4}_{x} L^{5k/2}_{(-\infty,0]}}\leq \delta,
$$
where $\delta=\delta((2A_C)^{1/2})$ is given by Proposition \ref{SDGT} (Small Data Theory) and then
$$
\|\textrm{KdV}(t)u_C(t_n)\| _{L^{5k/4}_{x} L^{5k/2}_{(-\infty,0]}}\leq 2\delta.
$$
Note that $\textrm{KdV}(t)u_C(t_n)=u_C(t+t_n)$, in particular, the previous inequality can be rewritten as
$$
\|u_C\| _{L^{5k/4}_{x} L^{5k/2}_{(-\infty,t_n]}}\leq 2\delta.
$$
Since $t_n\rightarrow +\infty$ we finally have 
\begin{equation*}
\|u_C\| _{L^{5k/4}_{x} L^{5k/2}_t}\leq 2\delta,
\end{equation*}
 which is a contradiction. 
 
On the other hand, if $t_n^1\rightarrow +\infty$ we use a similar argument  to deduce, for $n$ large, that 
\begin{equation*}
\|u_C\| _{L^{5k/4}_{x} L^{5k/2}_{[t_n,+\infty)}}=\|\textrm{KdV}(t)u_C(t_n)\| _{L^{5k/4}_{x} L^{5k/2}_{[0,+\infty)}}\leq 2\delta.
\end{equation*}
Now, applying Remark \ref{Rema-SDT} we have $\|u_C\| _{L^{5k/4}_{x} L^{5k/2}_{[0,+\infty)}}<\infty$, which is also a contradiction. 

So, we can assume that, up to a subsequence, there exists $t^1\in \R$ such that $t_n^1\rightarrow t^1$, as $n\rightarrow \infty$. Moreover, from \eqref{ucdecom}, \eqref{Wn} and the continuity of the linear flow we have
$$
u_{C}(t_n,\cdot+x_n) \rightarrow U(t^1)\psi \quad \textrm{in} \quad H^1, \quad \mbox{as}\;\; n\rightarrow \infty.
$$

In the case $\|u_C\|_{L^{5k/4}_xL^{5k/2}_{(-\infty,0]}}=\infty$ a similar proof leads to the precompactness of $\{u_C(t,\cdot-x(t)): t\leq 0\}$ and this completes the proof of Proposition \ref{Precomp}.
\end{proof}

%%%%%%%%%%%%%%%%%%%%%%%%%%%%%%%%%%%%%%%%%%%%%%%%%%%%%%%%%%%%%%%%%%%%%%%%%%%%%%%%%%%%%%%%%%%%%%%%%%%%%%%%%%%%%%%%%%%%%%%%%%%%
%%%%%%%%%%%%%%%%%%%%%%%%%%%%%%%%%%%%%%%%%%%%%%%%%%%%%%%%%%%%%%%%%%%%%%%%%%%%%%%%%%%%%%%%%%%%%%%%%%%%%%%%%%%%%%%%%%%%%%%%%%%%
%%%%%%%%%%%%%%%%%%%%%%%%%%%%%%%%%%%%%%%%%%%%%%%%%%%%%%%%%%%%%%%%%%%%%%%%%%%%%%%%%%%%%%%%%%%%%%%%%%%%%%%%%%%%%%%%%%%%%%%%%%%%
%%%%%%%%%%%%%%%%%%%%%%%%%%%%%%%%%%%%%%%%%%%%%%%%%%%%%%%%%%%%%%%%%%%%%%%%%%%%%%%%%%%%%%%%%%%%%%%%%%%%%%%%%%%%%%%%%%%%%%%%%%%%
%%%%%%%%%%%%%%%%%%%%%%%%%%%%%%%%%%%%%%%%%%%%%%%%%%%%%%%%%%%%%%%%%%%%%%%%%%%%%%%%%%%%%%%%%%%%%%%%%%%%%%%%%%%%%%%%%%%%%%%%%%%%
%%%%%%%%%%%%%%%%%%%%%%%%%%%%%%%%%%%%%%%%%%%%%%%%%%%%%%%%%%%%%%%%%%%%%%%%%%%%%%%%%%%%%%%%%%%%%%%%%%%%%%%%%%%%%%%%%%%%%%%%%%%%
%%%%%%%%%%%%%%%%%%%%%%%%%%%%%%%%%%%%%%%%%%%%%%%%%%%%%%%%%%%%%%%%%%%%%%%%%%%%%%%%%%%%%%%%%%%%%%%%%%%%%%%%%%%%%%%%%%%%%%%%%%%%
%%%%%%%%%%%%%%%%%%%%%%%%%%%%%%%%%%%%%%%%%%%%%%%%%%%%%%%%%%%%%%%%%%%%%%%%%%%%%%%%%%%%%%%%%%%%%%%%%%%%%%%%%%%%%%%%%%%%%%%%%%%%
%%%%%%%%%%%%%%%%%%%%%%%%%%%%%%%%%%%%%%%%%%%%%%%%%%%%%%%%%%%%%%%%%%%%%%%%%%%%%%%%%%%%%%%%%%%%%%%%%%%%%%%%%%%%%%%%%%%%%%%%%%%%

\section{Rigidity theorem and extinction of the critical solution}\label{sec6}

In this section we will prove that the critical solution constructed in Section \ref{sec5} cannot exist (see Proposition \ref{finalprop}), which in turn will complete the proof of Theorem \ref{global4}. Our argument is based on the Kenig-Merle compactness/rigidity method (see, for instance, \cite{kenig-merle} and \cite{kenig-merle2}). In particular the rigidity is obtained due to a suitable version of the interaction Morawetz type estimates
adapted to the gKdV equation.

We introduce the following densities:
\begin{equation}\label{densities}
\begin{cases}
\rho(t, x)=(u(t,x))^2\\
e(t, x)=\frac{1}{2} |\partial_x u(t, x)|^2 + \frac 1{k+2} (u(t,x))^{k+2}
\\j(t,x)=3 (\partial_x u(t,x))^2+\frac{2(k+1)}{k+2} (u(t,x))^{k+2}
\\k(t,x)=\frac 32 (\partial_{xx} u(t,x))^2
+2(\partial_x u(t,x))^2 (u(t,x))^{k}+\frac 12 (u(t,x))^{2k+2},
\end{cases}
\end{equation}
which are well defined provided that $u(t)\in H^2$.
Then we get the following relations:
\begin{align}\label{transp}
\partial_t \rho + \partial_{xxx} \rho&= \partial_x j,\\\nonumber
\partial_t e + \partial_{xxx} e&= \partial_x k.
\end{align}
Moreover we have the following key property (see \cite[Theorem 2]{tao1}):
\begin{equation}\label{keyobservation}
\Big (\int \rho(x) dx \Big) \cdot \Big( \int k(x) dx\Big)> 
\Big( \int e(x) dx\Big) \cdot \Big( \int 
j(x) dx\Big)
\end{equation}
where $\rho(x), e(x), j(x), k(x)$ are respectively 
the expressions above at time $t=0$, namely 
$\rho(0,x), e(0,x), j(0,x), k(0,x)$ and provided that $u(0, x)=u_0(x)\neq 0$.

\subsection{Local smoothing and continuity of the flow}

In the next sections we shall use the following facts which are well known
for the flow associated with \eqref{gkdv}. Recall we denote by $\KdV(t)$ the flow associated 
with \eqref{gkdv} which, by Theorem \ref{global5}, is well-defined globally in time in the space $H^1$.
Next we describe other key properties of the flow map that will be crucial in the sequel.

\begin{proposition}\label{regprop}
For every $t\in \mathbb R$ we have $\KdV(t) \in  C(H^s(\mathbb R); H^s(\mathbb R))$, $s=1,2$. Moreover we have the following 
local smoothing property,
$$
u_0\in H^1 \Rightarrow \KdV(t)u_0\in H^2_{loc} \hbox{ a.e. } t\in \mathbb R.
$$ 
\end{proposition}
\begin{proof}
The continuity of the flow of $\KdV(t)$ is proved for instance in \cite{kpv1}. The second affirmation follows the same argument given
by Kato in \cite{K83} and the local results in \cite{kpv1}.
\end{proof}

\subsection{Interactive a priori estimates}

Here we give an interactive a priori estimate concerning the densities introduced above.

\begin{lemma} Let us fix
	$Q\in C^\infty(\mathbb R)$ such that
	\begin{itemize}
		%\item $Q(x)=x \quad \forall |x|<1$;
		\item $Q'(x)\geq 0$;
		%\item $\lim_{x\rightarrow \pm \infty} Q(x)= \pm 2$;
		\item $Q'(x)=0$ for $|x|>R_0$.
	\end{itemize}
	Assume  that $u(t)$ solves \eqref{gkdv} with initial
	condition $u_0\in H^1$, then for any $t_1<t_2$, we have
	\begin{align}\label{ideens}
	&\int\int Q(x-y) \rho(t_2, x)e(t,y) dx dy- 
	\int\int Q(x-y) \rho(t_1, x)e(t,y) dx dy \geq \\\nonumber
	&-\int_{t_1}^{t_2} \int\int Q^{'}(x-y)  j(t, x)e(t,y) dx dy dt+ 
	\int_{t_1}^{t_2}\int\int Q^{'}(x-y) \rho(t, x) k(t,y) dx dydt,
	\end{align}
	where on the r.h.s. we assume to be equal to $+\infty$ 
	for every $t$ such that $u(t)\notin H^2_{loc}$.
\end{lemma}

\begin{remark}
	Let us notice that the l.h.s. as well as all the terms of the r.h.s., except 
	the one involving $Q^{'}(x-y) \rho(t, x) k(t,y)$, are well-defined for every solution
	belonging to $H^1$, since they do not involve second derivatives of the solution.
	On the contrary, the aforementioned term involves the term $|\partial_{xx} u|^2$
	and hence in principle the quantity is not necessarily finite.
	As a consequence of the estimate above we deduce that in fact necessarily, for almost every time we have that  $\int \int Q^{'}(x-y) \rho(t, x) k(t,y) dx dy<\infty$.
	In fact this is something well--known due to the local smoothing effect on the gain of one derivative locally in space.
	
\end{remark}

\begin{proof}
First we assume $u_0\in H^2$ and hence by the persistence property (see Proposition \ref{regprop}) we have
	$u(t)\in H^2$. In this situation
	we can justify all the computations that we write below and we get a stronger version
	of the desired estimate, since we get an identity.
	Notice that we have
	$$\int\int Q(x-y) \rho(t, x)e(t,y) dx dy$$
	and, by \eqref{transp}, we get
	\[
	\begin{split}
	\frac d{dt} 
	\int&\int Q(x-y) \rho(t, x)e(t,y) dx dy\\
	= 
	&\int\int Q(x-y) \partial_t \rho(t, x)e(t,y) dx dy+ 
	\int\int Q(x-y) \rho(t, x)\partial_t e(t,y) dx dy\\
	= 
	&-\int\int Q(x-y) \partial_{xxx} \rho(t, x)e(t,y) dx dy- 
	\int\int Q(x-y) \rho(t, x)\partial_{yyy} e(t,y) dx dy
	\\
	&+\int\int Q(x-y) \partial_x j(t, x)e(t,y) dx dy+ 
	\int\int Q(x-y) \rho(t, x)\partial_y k(t,y) dx dy
	\end{split}
	\]
	and, by integration by parts,
		\[
	\begin{split}
	\frac d{dt} 
	&\int\int Q(x-y) \rho(t, x)e(t,y) dx dy\\ 
	&=-\int\int Q^{'}(x-y)  j(t, x)e(t,y) dx dy+ 
	\int\int Q^{'}(x-y) \rho(t, x) k(t,y) dx dy.
	\end{split}
\]
We conclude by integration w.r.t.  $t$ variable.
	
	Next we work by density, more precisely given $u_0\in H^1$ choose
	$u_{n,0}\in H^2$ such that $u_{n,0}\rightarrow u_0$ in $H^1$.
	By continuous dependence we have 
	$u_n(t)\rightarrow u(t)$ in $H^1$ for every $t\in \mathbb R$,
	where $u_n(t)$ and $u(t)$ are, respectively, the solutions
	to \eqref{gkdv} associated with $u_{n,0}$ and $u_0$.
	Next recall that by the previous step we have that 
	\eqref{ideens} is satisfied under the stronger form of an identity provided that we replace
	$\rho(t,x), e(t,x), j(t,x), k(t,x)$ by $\rho_n(t,x), e_n(t,x), j_n(t,x), k_n(t,x)$,
	which are the corresponding densities computed along the solutions $u_n(t,x)$.
	It is easy to check that we can pass to the limit in all the terms involved 
	in the identity except the term 
	$$\int_{t_1}^{t_2}\int\int Q^{'}(x-y) \rho_n(t, x) k_n(t,y) dx dy dt$$
	since it involves two derivatives of $u_n$ and hence the convergence
	of $u_n(t)$ to $u(t)$ in $H^1$ does not allows to pass to the limit.
	
	Nevertheless,  to conclude the desired inequality \eqref{ideens}, it is sufficient to show that
	\begin{equation*}
	\begin{split}
	\liminf_{n\rightarrow \infty} & \int_{t_1}^{t_2}\int\int Q^{'}(x-y) \rho_n(t, x) k_n(t,y) dx dy dt\\
	\geq  &\int_{t_1}^{t_2}\int\int Q^{'}(x-y) \rho(t, x) k(t,y) dx dy dt 
	\end{split}
	\end{equation*}
	and in fact (by looking at the expression of $k(t,x)$ and $\rho(t,x)$ and by recalling that
	we can pass to the limit along all the expressions that do not involve 
	second derivatives) it is sufficient to show
	\begin{equation*}
	\begin{split}
	\liminf_{n\rightarrow \infty}& \int_{t_1}^{t_2}\int\int Q^{'}(x-y) |u_n(t, x)|^2 |\partial_{xx} u_n(t,y)|^2 dx dy dt\\
	\geq &\int_{t_1}^{t_2}\int\int Q^{'}(x-y) |u(t, x)|^2 |\partial_{xx} u(t,y)|^2 dx dy dt.
	\end{split}
	\end{equation*}
	By Fatou's lemma w.r.t. to time integration it is sufficient to show that
	\begin{align}\label{prprim}
	\liminf_{n\rightarrow \infty} 
	\int\int Q^{'}(x-y) |u_n(t, x)|^2 |\partial_{xx} u_n(t,y)|^2 dx dy 
	\\\nonumber\geq \int\int Q^{'}(x-y) |u(t, x)|^2 |\partial_{xx} u(t,y)|^2 dx dy,
	\hbox{ a.e. } t\in \mathbb R.
	\end{align}
	To prove this fact we use again Fatou's lemma w.r.t.  $x$ variable and hence it is sufficient to show that 
	\begin{align}\label{seconded}
	 \liminf_{n\rightarrow \infty}
	&\int Q^{'}(x-y) |\partial_{xx} u_n(t,y)|^2 dy
	\\\nonumber & \geq \int Q^{'}(x-y) |\partial_{xx} u(t,y)|^2 dy, \quad \forall x\in \mathbb R, \hbox{ a.e. } t\in 
	\mathbb R.
	\end{align}
	Indeed, by combining this estimate with the fact that 
	$|u_n(t,x)|^2\rightarrow |u(t, x)|^2$ for $a.e.\; x\in \mathbb R$ (that follows from 
	$u_n(t,x)\rightarrow u(t,x)$ in $H^1$), 
	we can conclude \eqref{prprim} by Fatou's lemma w.r.t. $x$ variable.
	Next we prove \eqref{seconded}. 
	We choose $t\in \mathbb R$ such that $u(t)\in H^2_{loc}(\R)$ (as a consequence of Proposition \ref{regprop} it can be done for a.e. 
	$t\in \R$), then we have two possibilities: either 
	$$
	\liminf_{n\rightarrow \infty}\int Q^{'}(x-y) |\partial_{yy}u_n(t,y)|^2 dy=\infty
	$$
	and \eqref{seconded} is trivial, or we have 
	$$
	\liminf_{n\rightarrow \infty}\int Q^{'}(x-y) |\partial_{yy}u_n(t,y)|^2 dy<\infty.
	$$
	Hence $u_n(t,x)\rightarrow u(t,x) \hbox{ in } H^1$ implies 
	$$ 
	\sqrt {Q^{'}(x-y)} \partial_{yy} u_n(t,y)\rightarrow \sqrt {Q^{'}(x-y)} \partial_{yy} u(t,y), \hbox{ weakly in } L^2(B_R),
	$$
	from which we deduce \eqref{seconded}.
	\end{proof}

\subsection{Additional properties of the critical solution}

Next we shall make use of the cut-off function
$\varphi_R(x)=\varphi\big( \frac xR \big)$, where
$\varphi(x)\in C^\infty_0 (\mathbb R), \varphi(x)\in [0,1], \varphi(x)=0$ for $|x|>1$ and 
$\varphi(x)=1$ for $|x|<1/2$.

\begin{lemma}\label{fipa}
	Assume that
	$u\in C(\mathbb R; H^1(\mathbb R)\setminus\{0\})$ is a solution of \eqref{gkdv} such that
	\begin{equation}\label{lowerbound}
	\inf_{t\in\R} \|u(t)\|_{L^2}>0;
	\end{equation}
	\begin{equation*}
	u(t)\in H^2_{loc} \hbox{ a.e. } t\in \mathbb R;
	\end{equation*}
	\begin{equation}\label{hyp}\exists x(t)\in \mathbb R \hbox{ s.t. } \{u(t, \cdot-x(t)), t\geq0
	\}\subset H^1
	\hbox{ is precompact in } H^1.
	\end{equation} Then 
	we have the following vanishing property
	\begin{equation}\label{van}
	\lim_{R\rightarrow \infty} \Big(\sup_{t\geq0}
	\int_{|x-x(t)|>R} |\partial_x u(t,x)|^2 + |u(t,x)|^2 + |u(t,x)|^{p} dx\Big)=0
	\end{equation}
	for every fixed $2\leq p<\infty$.
	Moreover, there exist
	$\bar \epsilon>0$ and $\bar R>0$ such that
	\begin{equation}\label{keyobservation1}
	\begin{split}
	&\Big(\int \rho_{R,x(t)}(t,x) dx \Big) \cdot \Big( \int k_{R,x(t)}(t,x) dx\Big)
	\\ -
	\Big( \int e_{R,x(t)}(t,x) dx\Big) & \cdot \Big( \int
	j_{R,x(t)}(t,x) dx\Big)>\bar \epsilon, \quad \forall R>\bar R, \quad \hbox{ a.e. } t\geq0,
	\end{split}
	\end{equation}
	where $\rho_{R,x(t)}(t,x), j_{R,x(t)}(t,x), e_{R,x(t)}(t,x), k_{R,x(t)}(t,x)$ are the densities \eqref{densities} computed on the functions
	$\varphi_R(x-x(t)) u(t,x)$. 
\end{lemma}

\begin{remark}
If instead of \eqref{hyp} we assume
\begin{equation}\label{hyp1}
\exists x(t)\in \mathbb R \hbox{ s.t. } \{u(t, \cdot-x(t)), t\leq0
\}\subset H^1
\hbox{ is precompact in } H^1
	\end{equation}
	then a similar conclusion is true. Since, from Proposition \ref{Precomp} we have that \eqref{hyp} or \eqref{hyp1} hold, in what follows we will suppose, without loss of generality, that \eqref{hyp} holds.
\end{remark}

\begin{proof}[Proof of Lemma \ref{fipa}] We prove the result by assuming $x(t)=0$, then the general case	is obtained by repeating the same argument 	up to space translation. In particular we shorten the notations as follows: $\rho_{R,0}(t,x)$, $j_{R,0}(t,x)$, $e_{R,0}(t,x)$, $k_{R,0}(t,x)$ will be denoted by $\rho_{R}(t,x)$, $j_{R}(t,x)$, $e_{R}(t,x)$, $k_{R}(t,x)$.	We prove first \eqref{van}.	Assume by the absurd that 
$$\exists t_n\geq0, R_n\rightarrow \infty, \bar \epsilon>0$$	such that
\begin{equation}\label{hgh}
\int_{|x|>R_n}\Big( |\partial_x u(t_n,x)|^2 + |u(t_n,x)|^2 + |u(t_n,x)|^{p}\Big)
	dx>\bar \epsilon
	\end{equation}
	In fact, if it is not the case then we deduce by  the compactness assumption 
	$u(t_n)\rightarrow \bar u\neq 0$ in $H^1$ (notice that $\bar u \neq 0$ since we are assuming $u(t)\neq 0$ and $u(t)$ precompact). 
	In particular since $\bar u \neq 0$ there exists $\bar R>0$ such that
	$$
	\int_{|x|>\bar R} |\partial_x \bar u(x)|^2 + |\bar u(x)|^2 + |\bar u(x)|^{p}
	dx<\frac{\bar \epsilon}2
	$$ and hence
	the same property remains true for $u(t_n)$, for $n$ large enough. This gives a contradiction
	with \eqref{hgh}.

	Next we focus on \eqref{keyobservation1}.
	Assume by the absurd that it is not true. Then 
	$$\exists t_n\geq0, R_n\rightarrow \infty, 
	\epsilon_n\rightarrow 
	0$$ such that
	%First of all notice that by \eqref{key observation} we have that 
	%the l.h.s. in \eqref{keyobservation1} is non-negative
\begin{equation}\label{nons}0
	<\big (\int \rho_{R_n}(t_n, x) dx \big) \cdot \big( \int
	k_{R_n}(t_n,x) dx\big)
 -\big( \int e_{R_n} (t_n, x) dx\big) \cdot \big( \int
 j_{R_n} (t_n, x) dx\big)\leq \epsilon_n,
	\end{equation}
	where the l.h.s.  inequality comes from \eqref{keyobservation}.
	Notice also that we can assume \begin{equation}\label{bounH2}
	\sup_n \|\varphi_{R_n} u(t_n)\|_{H^2}<\infty.\end{equation}
	In fact if it is not the case then, by looking at the expression of  $k_{R_n}(t_n,x)$,
	by the boundedness of  $u(t)$ in $H^1$ and 
	recalling \eqref{lowerbound}, we get
	$$\sup_n \int e_{R_n} (t_n, x) dx, \int
	j_{R_n} (t_n, x) dx<\infty,\quad \hbox{and} \quad \inf_n \int \rho_{R_n} (t_n, x) dx>0$$ and hence 
	\begin{align*}&
	\lim_{n\rightarrow \infty} \big (\int \rho_{R_n}(t_n, x) dx \big) \cdot \big( \int
	k_{R_n}(t_n,x) dx\big)-
	\big( \int e_{R_n} (t_n, x) dx\big) \cdot \big( \int
	j_{R_n} (t_n, x) dx\big)=\infty,
	\end{align*}
	which is a contradiction with \eqref{nons}.
	On the other hand by using the compactness assumption \eqref{hyp}
	we get the existence of $\bar u \in H^1$ such that
	$u(t_n)\rightarrow \bar u \neq 0$ strongly in $H^1$ and also
	by \eqref{bounH2} we deduce $\varphi_{R_n} u(t_n)\rightharpoonup \bar u$ weakly in $H^2$.
	Therefore, by the estimate \eqref{nons} (and by using the weak-semicontinuity of the $H^2$ norm) 
	we get 
	\begin{align*}0&<\big (\int \bar \rho(x) dx \big) \cdot \big( \int
	\bar k(x) dx\big) -
	\big( \int \bar e(x) dx\big) \cdot \big( \int
	\bar j (x) dx\big)\leq \epsilon_n,
	\end{align*}
	where $\bar \rho, \bar k, \bar e, \bar j$ are respectively the densities
	$\rho,k, e,j$ (see \eqref{densities}) computed on the function $\bar u$, and again on the l.h.s. we have used \eqref{keyobservation}. 
	We get a contradiction by taking the limit as $n\rightarrow \infty$.
	\end{proof}

As a consequence we get the following

\begin{lemma} Let $u$ be as in Lemma \ref{fipa}, then there exist
	$\bar \epsilon>0$ and $\bar R>0$ such that
	\begin{align}\label{keyobservation2}
	&\big (\int_{|x-x(t)|<R}  \rho(t,x) dx \big) \cdot \big( \int_{|x-x(t)|<R} k(t,x) dx\big)
	\\\nonumber -
	\big( \int_{|x-x(t)|< R} e(t,x) dx\big) & \cdot \big( \int_{|x-x(t)|< R}
	j(t,x) dx\big)>\bar \epsilon, \quad \forall R>\bar R, \quad \hbox{ a.e. } t\geq0,
	\end{align}
	where $\rho(t,x), e(t,x), k(t,x), j(t,x)$ are respectively the densities
	\eqref{densities} computed on the function $u(t,x)$. 
\end{lemma}
\begin{proof} We use the same notations as in Lemma \ref{fipa}.
	We split \begin{equation}\label{alb}\int k_{R, x(t)}(t,x) dx=  \frac 32
	\int  |\partial_{xx} (\varphi_{R, x(t)} (x) u(t, x))|^2 dx
	+ \int \tilde k_{R, x(t)}(t,x) dx\end{equation}
	where $\tilde k_{R, x(t)}(t,x)$ is the density $$
	\tilde k (\psi(x))=2(\partial_x \psi(x))^2 (\psi(x))^{k}+\frac 12 (\psi(x))^{2k+2}$$
	computed along the function $\psi(x)=\varphi_{R, x(t)} (x) u(t, x)$.
	
	By using \eqref{van}, we infer
	\begin{equation}\label{alb2}\int \tilde k_{R, x(t)}(t,x) dx - \int_{|x-x(t)|<R} \tilde k(t,x)dx=o_R(1),
	\end{equation}
	where $\lim_{R\rightarrow \infty} o_R(1)=0$ uniformly w.r.t. to $t$
	and  $\tilde k(t,x)=\tilde k(u(t,x))$.
	
	By combining again \eqref{van}, with the cut--off properties of $\varphi_R(x)$ and 
	with the uniform boundedness of $u(t)$ in $H^1$, we obtain
	$$\frac 32 \int  |\partial_{xx} (\varphi_{R, x(t)} (x) u(t, x))|^2 dx
	-\frac 32  \int  |\varphi_{R, x(t)} (x) \partial_{xx} u(t, x))|^2 =o_R(1).$$
	By summing this identity with the following trivial one
	$$\frac 32 \int_{|x-x(t)|<R}  |\partial_{xx} u(t, x))|^2 dx -  
	\frac 32 \int_{|x-x(t)|<R }  |\partial_{xx} u(t, x))|^2 dx =0,
	$$
	we get
	\begin{equation}\label{alb3}\frac 32 \int  |\partial_{xx} (\varphi_{R, x(t)} (x) u(t, x))|^2 dx 
	-  
	\frac 32 \int_{|x-x(t)|<R}  |\partial_{xx} u(t, x))|^2 dx\leq o_R(1),
	\end{equation}
	where we have used the properties of $\varphi$ that guarantee 
	$$-\frac 32  \int  |\varphi_{R, x(t)} (x) \partial_{xx} u(t, x))|^2 +
	\frac 32 \int_{|x-x(t)|<R }  |\partial_{xx} u(t, x))|^2 dx \geq 0.$$
	The proof follows by  combining \eqref{keyobservation1}
	with \eqref{alb}, \eqref{alb2}, \eqref{alb3}
	and with the following facts that in turn come from \eqref{van}:
	\begin{align*}&\int j_{R, x(t)}(t,x) dx - \int_{|x-x(t)|<R} j(t,x) dx=o_R(1)
	\\
	\nonumber
	&\int e_{R, x(t)}(t,x) dx - \int_{|x-x(t)|<R} e(t,x) dx
	= o_R(1),
	\\
	\nonumber&\int \rho_{R, x(t)}(t,x) dx - \int_{|x-x(t)|<R} \rho(t,x) dx=o_R(1).
	\end{align*}
\end{proof}

Finally, we establish our main result is this section by showing that indeed the critical solution constructed in Section \ref{sec5} cannot exist.

\begin{proposition}\label{finalprop}
	There does not exist any nontrivial solution $u$ to \eqref{gkdv},
	that satisfies \eqref{hyp} for a suitable $x(t)\in \mathbb R$.
\end{proposition}
\begin{proof}
	We introduce a function $\Phi\in C^\infty(\mathbb R)$ such that:
	\begin{itemize}
		\item $\Phi(x)=x \quad \forall\; |x|<1$;
		\item $\Phi'(x)\geq 0$;
		\item $|\Phi'(x)|=0$ for $|x|>2$.
	\end{itemize}
	Next, we consider for $R>0$ the rescaled functions
	$Q_R(x)= {2R}\Phi(\frac{x}{2R})$
	%We consider the following quantity:
	%$$\int\int Q_R(x-y) \rho(t, x)e(t,y) dx dy$$
	%and we notice that by \eqref{transp} we get
	%\begin{align*}\frac d{dt} 
	%&\int\int Q_R(x-y) \rho(t, x)e(t,y) dx dy\\\nonumber 
	%= 
	%\int\int Q_R(x-y) \partial_t \rho(t, x)e(t,y) dx dy&+ 
	%\int\int Q_R(x-y) \rho(t, x)\partial_t e(t,y) dx dy\\\nonumber
	%= 
	%-\int\int Q_R(x-y) \partial_{xxx} \rho(t, x)e(t,y) dx dy&- 
	%\int\int Q_R(x-y) \rho(t, x)\partial_{yyy} e(t,y) dx dy
	%\\\nonumber 
	%+\int\int Q_R(x-y) \partial_x j(t, x)e(t,y) dx dy&+ 
	%\int\int Q_R(x-y) \rho(t, x)\partial_y k(t,y) dx dy
	%\end{align*}
	%and by integration by parts we get
	and  use \eqref{ideens} with $Q=Q_R$ to obtain
	\begin{align*}
	&\int\int Q_R(x-y) \rho(t_2, x)e(t_2,y) dx dy - 
	\int\int Q_R(x-y) \rho(t_1, x)e(t_1,y) dx dy dt \geq \\\nonumber
	&-\int_{t_1}^{t_2} \int\int Q_R^{'}(x-y)  j(t, x)e(t,y) dx dy dt + 
	\int_{t_1}^{t_2} \int\int Q_R^{'}(x-y) \rho(t, x) k(t,y) dx dy dt
	\\\nonumber & = \int_{t_1}^{t_2} (-I_R(t)+II_R(t)) dt.
	\end{align*}

First observe that
	$$I_R(t)= \int \int_{|x-y|<2R}  j(t, x)e(t,y) dx dy
	+ \int \int_{|x-y|>2R} Q_R^{'}(x-y) j(t, x)e(t,y) dx dy
	$$
	and
	$$II_{R}(t)=\int\int_{|x-y|<2R}  \rho(t, x) k(t,y) dx dy
	+ \int\int_{|x-y|>2R} Q_R^{'}(x-y)  \rho(t, x) k(t,y) dx dy.$$
	Next, for any fixed $R>0$ and $t\geq0$, we introduce
	%\begin{equation}\label{subset}
	%Q_{x(t),R}\subset \{(x,y)\in {\mathbb R}^2| |x-y|<\}
	%\end{equation}
	$$Q_{x(t),R}= [x(t)-R, x(t)+R]\times
	[x(t)-R, x(t)+R]$$
	and observe that
	\begin{align*} I_R(t)&= 
	\int \int_{Q_{x(t),R} }  j(t, x)e(t,y) dx dy\\\nonumber +
	\int \int_{\{|x-y|<2R\}\setminus Q_{x(t),R} }  & j(t, x)e(t,y) dx dy
	+ \int \int_{|x-y|>2R}  Q_R^{'}(x-y) j(t, x)e(t,y) dx dy.
	\end{align*}
	Now, by \eqref{van}, we have
	$$
	\lim_{R\rightarrow \infty}
	\int \int_{\{|x-y|<2R\}\setminus Q_{x(t),R} }  j(t, x)e(t,y) dx dy
	+ \int \int_{|x-y|>2R}  Q_R^{'}(x-y) j(t, x)e(t,y) dx dy=0.$$
	Hence,
	$$\lim_{R\rightarrow \infty} |I_R(t)- 
	\int \int_{Q_{x(t),R} }  j(t, x)e(t,y) dx dy|=0.$$
	Moreover,
	\begin{align*} II_R(t)=& 
	\int \int_{Q_{x(t),R} }  \rho(t, x)k(t,y) dx dy\\\nonumber& +
	\int \int_{\{|x-y|<2R\}\setminus Q_{x(t),R} }  \rho(t, x) k(t,y) dx dy
	\\\nonumber & + \int \int_{|x-y|>2R}  Q_R^{'}(x-y) \rho(t, x)k(t,y) dx dy
	\geq \int \int_{Q_{x(t),R} }  \rho(t, x)k(t,y) dx dy.
	\end{align*}
	Since, by \eqref{keyobservation2},
	$$\int \int_{Q_{x(t),R} }  \rho(t, x)k(t,y) dx dy- \int \int_{Q_{x(t),R} }  j(t, x)e(t,y) dx dy\geq \delta_0>0, \hbox{ a.e. } t\geq0,$$
	by combining the estimates above on $I_R(t)$ and $II_R(t)$
   we deduce
	the existence of $\delta_0>0$ and $R_0>0$ such that
	$$
	\int\int Q_R(x-y) \rho(t_1, x)e(t_1,y) dx dy
	- \int\int Q_R(x-y) \rho(t_2, x)e(t_2,y) dx dy $$$$\geq \delta_0 (t_2-t_1), \quad \forall R>R_0.$$
	
	By choosing $t_1=0$, $ t_2=T$, and $R=2R_0$, we obtain
	$$\int\int Q_{2R_0}(x-y) \rho(T, x)e(T,y) dx dy - 
	\int\int Q_{2R_0}(x-y) \rho(-T, x)e(-T,y) dx dy\geq  \delta_0 T
	$$
	and we get an absurd by taking the limit as $T\rightarrow \infty$,
	and by noticing that the l.h.s. is uniformly bounded with respect to $T$
	by the conservation of the energy.
	
\end{proof}

%\subsection*{Acknowledgment}
%The authors thank the referee for helpful comments and suggestions which improved the presentation of the paper.

\bibliographystyle{mrl}

\end{document}